\documentclass[11pt,a4paper,twoside]{amsart}

\usepackage{amsfonts, amsmath,amssymb}
\usepackage[all]{xy}
\usepackage[lite]{amsrefs}
\usepackage{amsthm}
\usepackage{geometry}
\usepackage{graphicx}
\usepackage{wrapfig}
\usepackage{subfig}
\usepackage{appendix}
\usepackage{pstricks, pstricks-add}

\newcommand{\Hy}{\mathcal{H}}
\newcommand{\ringO}{\mathcal{O}}
\newcommand{\PSLO}{\mathrm{PSL}_2\left(\mathcal{O}_{-m}\right)}
\newcommand{\SLO}{\mathrm{SL}_2\left(\mathcal{O}_{-m}\right)}

\newcommand{\C}{{\mathbb{C}}}
\newcommand{\R}{{\mathbb{R}}}
\newcommand{\Z}{{\mathbb{Z}}}
\newcommand{\F}{{\mathbb{F}}}

\newcommand{\N}{{\mathbb{N}}}
\newcommand{\rationals}{{\mathbb{Q}}}

\newcommand*{\Homol}{\operatorname{H}}

\newcommand{\Graph}{{\mathcal{G}}}

\renewcommand{\geq}{\geqslant}

\newcommand{\mat}{\begin{pmatrix}a & b\\c & d\end{pmatrix}}
\newcommand{\trace}{\mathrm{tr}}
\newcommand{\im}{\mathrm{image}}
\newcommand{\Afour}{\mathcal{A}_4}
\newcommand{\Sthree}{\mathcal{S}_3}
\newcommand{\Kleinfourgroup}{\mathcal{D}_2}

\theoremstyle{plain}
\newtheorem{thm}{\bfseries Theorem}
\newtheorem{theorem}[thm]{\bfseries Theorem}
\newtheorem{Lem}[thm]{\bfseries Lemma}
\newtheorem{lemma}[thm]{\bfseries Lemma}
\newtheorem{sublem}[thm]{\bfseries Sub-Lemma}
\newtheorem{prop}[thm]{\bfseries Proposition}
\newtheorem{proposition}[thm]{\bfseries Proposition}
\newtheorem{cor}[thm]{\bfseries Corollary}
\newtheorem{corollary}[thm]{\bfseries Corollary}

\newtheorem{df}[thm]{\bfseries Definition}

\theoremstyle{remark}

\newtheorem{observation}[thm]{\bfseries Observation}

\newtheorem{rem}[thm]{\bfseries Remark}

\newtheorem{notation}[thm]{\bfseries Notation}

\newcommand{\circlegraph}{
\begin{pspicture}       (-0.3,-0.3)(0.3,0.3)
                        \pscircle(0,0.0){0.25}
                        \psdots(0.25,0.0)
\end{pspicture} }

\newcommand{\edgegraph}{ 
\begin{pspicture}(-0.3,-0.3)(0.3,0.3)
        \psdots(-0.2,0.0)
        \psdots(0.2,0.0)
        \psline(-0.2,0.0)(0.2,0.0)
\end{pspicture} }

\newcommand{\graphFive}{  
\begin{pspicture}(-0.3,-0.3)(0.3,0.3)
        \pscircle(0,0.0){0.25}
        \psdots(0,-0.25)
        \psdots(0,0.25)
        \psline(0,-0.25)(0,0.25)
\end{pspicture} }

\newcommand{\graphTwo}{  
\begin{pspicture}(-0.3,0.5)(0.85,1.1)
        \pscircle(0,0.8){0.25}
        \psdots(0.25,0.8)
        \psline(0.25,0.8)(0.8,0.8)
        \psdots(0.8,0.8)
\end{pspicture} }

\begin{document}

\title[The homological torsion of PSL$_2$ of the imaginary quadratic integers]{The homological torsion \\ of PSL$_2$ of the imaginary quadratic integers}
\author{Alexander D. Rahm}
\email{Alexander.Rahm@Weizmann.ac.il}
\urladdr{http://www.wisdom.weizmann.ac.il/\char126rahm/}
\address{Department of Mathematics, Weizmann Institute of Science, Rehovot, Israel}
\curraddr{Department of Mathematics, National University of Ireland at Galway}
\date{\today}
\maketitle

\begin{abstract}
The \emph{\mbox{Bianchi} groups} are the groups (P)$\mathrm{SL_2}$ over a ring of integers in an imaginary quadratic number field.
We reveal a correspondence between the homological torsion of the \mbox{Bianchi} groups and new geometric invariants, which are effectively computable thanks to their action on hyperbolic space.
We expose a novel technique, the \emph{torsion subcomplex reduction}, to obtain these invariants.
We use it to explicitly compute the integral group homology of the Bianchi groups.

Furthermore, this correspondence facilitates the computation of the equivariant $K$-homology of the Bianchi groups. 
By the Baum/Connes conjecture, which is satisfied by the \mbox{Bianchi} groups, we obtain the $K$-theory of their reduced $C^*$-algebras in terms of isomorphic images of their equivariant $K$-homology.
\end{abstract}

\section{Introduction}

Denote by $\rationals(\sqrt{-m})$, with $m$ a square-free positive integer, an imaginary quadratic number field, and by $\ringO_{-m}$ its ring of integers. 
The \emph{\mbox{Bianchi} groups} are the groups (P)$\mathrm{SL_2}(\ringO_{-m})$.
The \mbox{Bianchi} groups may be considered as a key to the study of a larger class of groups, the \emph{Kleinian} groups, which date back to work of Henri Poincar\'e~\cite{Poincare}.
In fact, each non-cocompact arithmetic Kleinian group is commensurable with some \mbox{Bianchi} group~\cite{MaclachlanReid}.
A wealth of information on the \mbox{Bianchi} groups can be found in the monographs \cites{Fine, ElstrodtGrunewaldMennicke, MaclachlanReid}.
These groups act in a natural way on hyperbolic three-space, which is isomorphic to the symmetric space associated to them.
The kernel of this action is the centre $\{ \pm 1 \}$ of the groups.
Thus it is useful to study the quotient of $\mathrm{SL_2}(\ringO_{-m})$ by its centre, namely $\mathrm{PSL_2}(\ringO_{-m})$.
In 1892, Luigi \mbox{Bianchi}~\cite{Bianchi} computed fundamental domains for this action when $m$ = 1, 2, 3, 5, 6, 7, 10, 11, 13, 15 and 19.
Such a fundamental domain has the shape of a hyperbolic polyhedron (up to a missing vertex at certain cusps, which represent the ideal classes of $\ringO_{-m}$), so we will call it the \emph{\mbox{Bianchi} fundamental polyhedron}. 
The computation of the \mbox{Bianchi} fundamental polyhedron has been implemented for all Bianchi groups~\cite{BianchiGP} in the language Pari/GP~\cite{Pari}. 

The images under $\mathrm{SL_2}(\ringO_{-m})$ of the facets of this polyhedron equip hyperbolic three-space with a cell structure.
In order to view clearly the local geometry, we pass to the \emph{refined cell complex}, which we obtain by subdiving this cell structure until the cell stabilisers fix the cells pointwise.
We will see how to exploit this cell complex in different ways, in order to see different aspects of the geometry of these groups.

 An essential invariant of groups is their homology (defined for instance in~\cite{Brown}). We can compute it for the Bianchi groups using the refined cell complex and the equivariant Leray/Serre spectral sequence which starts from the group homology of the stabilisers of representatives of the cells, 
and converges to the group homology of the \mbox{Bianchi} groups.
We will state in proposition~\ref{non-Euclidean results} the results for simple integer coefficients in the cases $m = 19,$ $43,$ $67$ and $163$, which are the non-Euclidean principal ideal domain cases. In contrast to these, the Euclidean principal ideal domain cases are already known from~\cite{SchwermerVogtmann}. For some results in class number 2, see~\cite{RahmFuchs}. For some results in cohomology, see \cite{BerkoveMod2}.

Throughout this article, we will use the ``number theorist's notation'' $\Z/n$ for the cyclic group of order $n$. The virtual cohomological dimension of the Bianchi groups is 2. In degrees strictly above 2, we express their homology in terms of the following Poincar\'e series at the primes $\ell = 2$ and $\ell = 3$: 
$$P^\ell_m(t) := \sum\limits_{q \thinspace = \thinspace 3}^{\infty} \dim_{\mathbb{F}_\ell} \Homol_q \left(\text{PSL}_2\bigl(\mathcal{O}_{-m}\bigr);\thinspace \Z/\ell \right)\thinspace t^q.$$ 
These two primes are the only numbers which occur as orders of non-trivial finite elements of ${\rm PSL}_2(\ringO_{-m})$. So it has been shown in \cite{RahmThesis} that the integral homology of these groups is, in degrees $q$ strictly above 2, a direct sum of copies of $\Z/2$ and $\Z/3$.

\begin{proposition} \label{non-Euclidean results}
The integral homology of ${\rm PSL}_2(\ringO_{-m})$, for \mbox{ $m \in \{ 19, 43, 67, 163\}$,}
is of isomorphism type $ \Homol_q({\rm PSL}_2(\ringO_{-m}); \thinspace \Z) \cong $
$ \begin{cases}
\Z^{\beta_1 -1} \oplus \Z/4\Z \oplus \Z/2\Z \oplus \Z/3\Z, & q = 2, \\
\Z^{\beta_1}, & q = 1,  \\ 
\end{cases} $ 
\\
where  \scriptsize
$\begin{array}{l|ccccc}
m        &  19 & 43 & 67 & 163 \\
\hline 
\beta_1  &  1   & 2  &  3 & 7   \\
\end{array}$ \normalsize
gives the Betti number $\beta_1$, and is in all higher degrees a direct sum of copies of $\Z/2 \Z$ and $\Z/3 \Z$, with the number of copies specified by the Poincar\'e series \begin{center}
$P^2_m(t) = \frac{-t^3(t^3 - 2t^2 + 2t - 3)}{(t-1)^2 (t^2 + t + 1 ) }$ and
$P^3_m(t) = \frac{-t^3(t^2 - t + 2)}{(t-1)(t^2+1)}$. \end{center}
\end{proposition}

We remark that in these four cases, the torsion in the integral homology of $ {\rm PSL}_2(\ringO_{-m})$ is of the same isomorphism type.
To understand this, we consider, for a prime $\ell$, the subcomplex of the orbit space consisting of the cells with elements of order $\ell$ in their stabiliser. We call it the \emph{$\ell$--torsion subcomplex}.
The following statement on how the homeomorphism type of the $\ell$--torsion subcomplex determines the equivariant spectral sequence is proven by the reduction of the torsion subcomplex carried out in~\cite{RahmThesis}.
This technique uses lemma \ref{geometricRigiditytheorem} to determine the possible type of stabiliser of a vertex $v$ with exactly two adjacent edges which have $\ell$--torsion in their stabilisers.
Then these two edges, together with $v$, are replaced by a single edge; and theorem~\ref{actionOnAxes} as well as some homological information about the finite groups in question are used to check that the induced morphisms on homology produce the same terms on the second page of the equivariant spectral sequence as before the replacement.

\begin{theorem} \label{HomologicalInvarianceUnderHomeomorphisms}
The $\ell$--primary part of the integral homology of ${\rm PSL}_2(\ringO_{-m})$ depends in degrees greater than $2$ (the virtual cohomological dimension) only on the homeomorphism type of the $\ell$--torsion subcomplex.
\end{theorem}

This theorem follows directly from the stronger lemmata~\ref{E2invarianceUnderHomeomorphisms} and~\ref{direct sum decomposition}.
\begin{figure} \caption{Results for the 3--torsion homology, expressed in $ P^3_m(t)$}

\label{3torsionsubgraphs} 

\small
\begin{tabular}{c|c|c}
$m$ specifying the \mbox{Bianchi} group  &  \begin{tabular}{c} {3-torsion subcomplex,} \\ {homeomorphism type} \end{tabular} 
& $ P^3_m(t)$ 
\\
\hline &  & \\
\begin{tabular}{c} 
2, 5, 6, 10, 11, 14, 15, 17, 22, 23,  \\
 29, 34, 35, 38, 41, 46, 47, 51, 55, 58, \\
 59, 71, 83, 87, 95, 115, 119, 123, 131, \\
 155, 159, 167, 179, 187, 191, 235, 267 \end{tabular}
& \circlegraph &$\frac{-2t^3}{t-1}$ 
\\&  &\\
26, 42, 143, 195 & \circlegraph \circlegraph & $2 \left(\frac{-2t^3}{t-1}\right)$ 
\\&  &\\
30, 107 & \circlegraph \circlegraph \circlegraph & $3 \left(\frac{-2t^3}{t-1} \right)$ 
\\ &  &\\
33 & \circlegraph \circlegraph \circlegraph \circlegraph & $4 \left(\frac{-2t^3}{t-1} \right)$ 
\\ &  &\\
\begin{tabular}{c} 
{1, 7, 19, 31, 43, 67, 79, 103,} \\
{127, 139, 151, 163, 199, 571  }\end{tabular} & 
\edgegraph
 & $\frac{-t^3(t^2 - t + 2)}{(t-1)(t^2+1)}$ 
\\&  &\\
13, 37, 91, 403, 427 &  
\edgegraph \edgegraph  & 
$2  \left( \frac{-t^3(t^2-t+2)}{(t-1)(t^2+1)} \right)$ 
\\&  &\\
39, 111, 183, 643 & \circlegraph \edgegraph & $\frac{-2t^3}{t-1} +\frac{-t^3(t^2 - t + 2)}{(t-1)(t^2+1)}$ 
\\ &  &\\
259 & \circlegraph \edgegraph \edgegraph & $\frac{-2t^3}{t-1} +2\left(\frac{-t^3(t^2 - t + 2)}{(t-1)(t^2+1)}\right)$
\\ &  &\\
21 & \circlegraph \circlegraph \edgegraph \edgegraph & $2 \left(\frac{-2t^3}{t-1}\right) +2\left(\frac{-t^3(t^2 - t + 2)}{(t-1)(t^2+1)}\right)$ 
\end{tabular}
\normalsize
\end{figure}
Examples for this theorem are given for the prime $\ell = 3$ and thirty-six Bianchi groups in figure~\ref{3torsionsubgraphs} (for $\ell = 2$, see \cite{RahmThesis}). In all the non-Euclidean principal ideal domain cases, the 2--torsion, and respectively 3--torsion subcomplexes are homeomorphic, which explains the results in proposition~\ref{non-Euclidean results}. 
Underlying theorem~\ref{HomologicalInvarianceUnderHomeomorphisms}, there is the following correspondence between the non-trivial cyclic subgroups of the vertex stabilisers and the geodesic lines around which they effect a rotation, and which we shall call \emph{rotation axes}.

\begin{theorem} \label{actionOnAxes}
For any vertex $v$ in hyperbolic space, the action of its stabiliser on the set of rotation axes passing through $v$, induced by the action of the \mbox{Bianchi} group, is equivalent to the conjugation action of this stabiliser on its non-trivial cyclic subgroups.
\end{theorem}

We shall give the proof of theorem \ref{actionOnAxes} in section \ref{The action of the Bianchi groups}.

\subsection{Organisation of the paper.}
We begin with generalities of the action of SL$_2(\C)$ on hyperbolic 3-space.
In section~\ref{The action of the Bianchi groups}, we examine
details that are specific to restricting the action to the subgroup $\SLO$.
This allows us to prove theorem~\ref{actionOnAxes}. 
In section~\ref{refined cell complex}, we describe the refined cellular complex: an $\SLO$--invariant cellular structure on hyperbolic 3-space, in which all cell stabilisation is pointwise; and we elaborate a method to check that this property holds.
In section~\ref{Rigidity of the action on the refined cell complex}, we establish the properties of the action on the refined cell complex that we will need in section~\ref{axisGraphReduction} in order to prove theorem~\ref{HomologicalInvarianceUnderHomeomorphisms}.
In section~\ref{Floege cellular complex}, we join some cusps to the refined cellular complex, and retract it equivariantly onto the 2-dimensional, co-compact Fl\"oge cellular complex.
We further explain in section~\ref{Equivariant Euler characteristic} how we use the equivariant Euler characteristic to check the correctness of quotient spaces computed for the cellular complexes.
In section~\ref{The maps induced on cohomology by finite subgroup inclusions}, we recollect some homological statements on the finite subgroups of the Bianchi groups.
In section~\ref{SpecSeq}, we recall a spectral sequence which we can compute with information on this complex, and which converges to the homology of the Bianchi group in question.
In sections~\ref{3primary} and~\ref{2primary}, we give a characterisation of the differentials on the first page of the spectral sequence.
We conclude our description of the spectral sequence with the statement that for $q \geq 3$, \quad $\Homol_q(\PSLO;\Z)$ is a direct sum of copies of $\Z/2$ and $\Z/3$.
In section~\ref{axisGraphReduction}, we examine the torsion subgraphs and prove a stronger version of theorem~\ref{HomologicalInvarianceUnderHomeomorphisms}.
In section~\ref{results}, we establish the results of proposition~\ref{non-Euclidean results}, and in section~\ref{Results for the homological torsion} those of figure~\ref{3torsionsubgraphs} as well as the corresponding table in 2-torsion.
Finally, we give results for the special linear groups in section~\ref{SL2}, and for equivariant $K$-homology and operator $K$-theory in section~\ref{K-theory}.
We append the construction of a free resolution for the alternating group on four objects in section~\ref{Appendix: The low terms of a free resolution for the alternating group on 4 objects}.

\subsection*{Acknowledgements}
The author would like to thank Graham Ellis for his helpful cooperation in the computations of section \ref{SL2}. He is grateful for discussions with Nicolas Bergeron, and especially for the careful lecture of the referee which led him to establish corollary~\ref{normaliser in the stabiliser} and lemma~\ref{sublemma}, replacing a previous, less conceptual proof of lemma~\ref{geometricRigiditytheorem}.
He would like to thank \mbox{Philippe Elbaz-Vincent} and the people acknowledged in \cite{RahmThesis} for their help.
This article is dedicated to the memory of Fritz Grunewald.

\section{The action on hyperbolic space} \label{geometricRigidity}

Consider hyperbolic three-space, for which we will use the upper-half space model $\Hy$.
As a set, $$ \Hy = \{ (z,\zeta) \in \C \times \R \medspace | \medspace \zeta > 0 \}. $$

It is diffeomorphic to the symmetric space of $G := \mathrm{SL}_2(\C)$, and the natural action of $G$ which we obtain this way on $\Hy$ can be expressed by the following formula of Poincar\'e \cite{Poincare}.
\\
For $\gamma = \scriptsize \mat \normalsize \in G$, the action of $\gamma$ on $\Hy$ is given by $\gamma \cdot (z,\zeta) = (z',\zeta')$, where
$$ \zeta' = \frac{|\det \gamma|\zeta}{|cz-d|^2 +\zeta^2|c|^2},$$
$$ z' = \frac{\left(\thinspace\overline{d -cz}\thinspace\right)(az-b) -\zeta^2\bar{c}a}{|cz-d|^2 +\zeta^2|c|^2}.$$

Let us recall  Felix Klein's classification of the elements in $\mathrm{SL}_2(\C)$, which passes to $\mathrm{PSL}_2(\C)$.
\begin{df}
An element $\gamma \in \mathrm{SL}_2(\C)$, $\gamma \neq \pm 1$, is called 
\emph{loxodromic} if its trace is not a real number. Otherwise, it is called 
$\begin{cases}
\rm parabolic, &  |\trace (\gamma)| = 2, \\
\rm hyperbolic, &  |\trace (\gamma)| > 2 \\
\rm elliptic, &  |\trace (\gamma)| < 2 
\end{cases}$
.
\end{df}
We find the geometric meaning of this classification in the following proposition, which summarizes some results of Felix Klein \cite{binaereFormenMathAnn9}, which are worked out in more detail in his lectures notes edited by Robert Fricke.

\begin{prop}[\cite{ElstrodtGrunewaldMennicke}] \label{EGM1.4}
Let $\gamma$ be a non-trivial element of $\mathrm{SL}_2(\C)$. Then the following holds:
\begin{itemize}
\item $\gamma$ is parabolic if and only if $\gamma$ has exactly one fixed point in $\partial \Hy$.
\item $\gamma$ is elliptic if and only if it has two fixed points in $\partial \Hy$ and if the points on the geodesic line in $\Hy$ joining these two points are also left fixed. The action of $\gamma$ is then a rotation around this line.
\item $\gamma$ is hyperbolic if and only if it has two fixed points in $\partial \Hy$ and if any circle in $\partial \Hy$ through these points together with its interior is left invariant. The line in $\Hy$ joining these two fixed points is then left invariant, but $\gamma$ has no fixed point in $\Hy$.
\item $\gamma$ is loxodromic in all other cases. The action of $\gamma$ has then two fixed points in $\partial \Hy$ and no fixed point in $\Hy$. The geodesic joining the two fixed points is the only geodesic in $\Hy$ which is left invariant.
\end{itemize}
\end{prop}
In \cite{Ratcliffe}, it is stated that the parabolic elements do not have a fixed point in the interior of~$\Hy$.
So by excluding the parabolic, hyperbolic and loxodromic cases, we obtain the following corollary concerning elliptic elements.

\begin{cor} \label{corollaryTo1.4}
Let $\gamma$ be a non-trivial element of $\mathrm{SL}_2(\C)$, admitting a fixed point $v \in \Hy$. Then $\gamma$ fixes pointwise a geodesic line through~$v$, and performs a rotation around this line.
\end{cor}

Proposition \ref{EGM1.4} and corollary \ref{corollaryTo1.4} again pass to $\mathrm{PSL}_2(\C)$, because the center $\{\pm 1\}$ \mbox{of $\mathrm{SL}_2(\C)$} acts trivially on $\Hy$.

\subsection{The action of the Bianchi groups} \label{The action of the Bianchi groups}
Let $m$ be a squarefree positive integer and $\rationals(\sqrt{-m}\thinspace)$ be an imaginary quadratic number field with ring of integers $\mathcal{O}_{-m}$. We restrict the above described action to the Bianchi group \mbox{$\Gamma := \mathrm{SL_2}(\mathcal{O}_{-m}) \subset \mathrm{SL}_2(\C)$}. 

We will work more closely with the geodesic lines of corollary \ref{corollaryTo1.4}, as they are central objects in the following statements on the geometry of the torsion elements of the Bianchi groups.
We will call the matrices $1$ and $-1$ in $\mathrm{SL}_2(\C)$ the \emph{trivially acting elements}, because they constitute the kernel of the action.
\begin{df}
We will call a geodesic line passing through the point $v \in \Hy$ a \emph{$\Gamma_v$--axis}, if there exists a non-trivially acting element of $\Gamma$ fixing this line pointwise.
\end{df}
We will call a group \emph{non-trivially acting} if it admits a non-trivially acting element. Denote by $\Gamma_v$ the stabiliser in $\Gamma$ of the cell $v$.
\begin{lemma} \label{axes-cyclicSubgroupsBijection}
For any vertex $v \in \Hy$, there is a bijection between the \mbox{$\Gamma_v$--axes} and the non-trivially acting cyclic subgroups of the stabiliser $\Gamma_v$.
It is given by associating to a \mbox{$\Gamma_v$--axis} the subgroup in $\Gamma$ of rotations around this axis.
\end{lemma}
\begin{proof}
First we show that any \mbox{$\Gamma_v$--axis} is attributed to some non-trivial cyclic subgroup of $\Gamma_v$. 
Let $l$ be a \mbox{$\Gamma_v$--axis,} and let $\dot{\Gamma}_l$  be the set of elements of $\Gamma$ fixing $l$ pointwise. 
It is a subset of $\Gamma_v$, because $\dot{\Gamma}_l$ fixes $l$ pointwise and thus fixes $v$.
It is a subgroup because the composites and inverses must again fix $l$ pointwise.
By the definition of the \mbox{$\Gamma_v$--axes,} this subgroup is non-trivially acting. 
 And it is cyclic because by corollary \ref{corollaryTo1.4},  $\dot{\Gamma}_l$  consists only of rotations around $l$.
\\
Now we show that any non-trivially acting cyclic subgroup of $\Gamma_v$ is attributed to some \mbox{$\Gamma_v$--axis.}
Let $\gamma$ be the generator of a non-trivially acting cyclic subgroup of $\Gamma_v$.
By corollary \ref{corollaryTo1.4}, there is a geodesic line containing $v$, around which $\gamma$ performs a rotation.
\end{proof}

\begin{proof}[Proof of theorem {\rm \ref{actionOnAxes}}]
Let $l$ be a \mbox{$\Gamma_v$--axis,} and $\gamma \in \Gamma_v$. Let $\dot{\Gamma}_l$ be the subgroup of $\Gamma$ fixing $l$ pointwise. Then $\gamma \cdot l$ is again a \mbox{$\Gamma_v$--axis;} and the subgroup of $\Gamma$ fixing $\gamma \cdot l$ pointwise is $\gamma \dot{\Gamma}_l \gamma^{-1}$. Hence by lemma \ref{axes-cyclicSubgroupsBijection}, we can transfer the action to $\Gamma_v$-conjugation of the nontrivially acting cyclic subgroups.
\end{proof}
We deduce the following corollary from theorem {\rm \ref{actionOnAxes}}.
\begin{corollary} \label{normaliser in the stabiliser}
Let $\gamma \in \Gamma_v$ be a non-trivially acting element, stabilising a vertex $v \in \Hy$.
Let~$l$ be the $\Gamma_v$--axis pointwise stabilised by~$\langle \gamma \rangle$.
Then the subgroup of~$\Gamma_v$ sending~$l$ to itself, is the normaliser of~$\langle \gamma \rangle$ in~$\Gamma_v$.
\end{corollary}

Lemma \ref{axes-cyclicSubgroupsBijection}, theorem \ref{actionOnAxes} and corollary~\ref{normaliser in the stabiliser} clearly pass from~$\Gamma = \mathrm{SL_2}(\mathcal{O}_{-m})$ to~$\Gamma = \mathrm{PSL_2}(\mathcal{O}_{-m})$.
\\
We will make use of the following list of isomorphism types of finite subgroups in the Bianchi groups, which has been established in \cite{SchwermerVogtmann} and follows directly from the classification in \cite{binaereFormenMathAnn9}.

\begin{Lem}[Klein] \label{finiteSubgroups}
The finite subgroups in $\mathrm{PSL}_2(\ringO)$ are exclusively of isomorphism types the cyclic groups of orders one, two and three, the Klein four-group \mbox{$\Kleinfourgroup \cong \Z/2 \times \Z/2$}, the symmetric group $\Sthree$ and the alternating group~$\Afour$. 
\end{Lem}

The stabilisers of the points inside $\Hy$ are finite and hence of the above-listed types.

\subsection{A cell complex for the Bianchi groups} \label{refined cell complex}

The Bianchi/Humbert theory \cites{Bianchi, Humbert} gives a fundamental domain for the action of $\Gamma$ on $\Hy$, which we shall call the \emph{Bianchi fundamental polyhedron}.
It is a polyhedron in hyperbolic space up to the missing vertex~$\infty$, 
and up to a missing vertex for each non-trivial ideal class if $\ringO_{-m}$ is not a principal ideal domain.
We observe the following notion of strictness of the fundamental domain: the interior of the Bianchi fundamental polyhedron contains no two points which are identified by~$\Gamma$.
Swan~\cite{Swan} proves a theorem which implies that the boundary of the Bianchi fundamental polyhedron consists of finitely many cells.
Swan further produces a concept for an algorithm to compute the Bianchi fundamental polyhedron.
Such an algorithm has been implemented by Cremona~\cite{Cremona} for the five cases where~$\ringO_{-m}$ is Euclidean, and by his students Whitley~\cite{Whitley} for the non-Euclidean principal ideal domain cases, Bygott~\cite{Bygott} for a case of class number 2 and Lingham (\cite{Lingham}, used in~\cite{CremonaLingham}) for some cases of class number 3; and finally Aran\'es~\cite{Aranes} for arbitrary class numbers.
Another algorithm based on this concept has independently been detailed in~\cite{RahmThesis} and implemented in~\cite{BianchiGP} for all Bianchi groups, so we can make explicit use of the Bianchi fundamental polyhedron.
We can check that the computed polyhedron is indeed a fundamental domain for~$\Gamma$ using the following observation of Poincar\'e~\cite{Poincare}: After a cell subdivision which makes the cell stabilisers fix the cells pointwise, the 2-cells (``faces'') of the fundamental polyhedron appear in pairs $(\sigma, \gamma \cdot \sigma)$ with $\gamma \in \Gamma$ --- so for every orbit of faces, we have exactly two representatives --- such that with the orientation for which the lower side of the face $\sigma$ lies on the polyhedron, the upper side of~$\gamma \cdot \sigma$ lies on the polyhedron.

We induce a cell structure on~$\Hy$ by the images under~$\Gamma$ of the faces, edges and vertices of the Bianchi fundamental polyhedron.

\subsubsection{ Pointwise stabilised cells} 
\label{properEdgeStabilization}

In order to view clearly the local geometry, we pass to the \emph{refined cell complex}, which we obtain by subdiving this cell structure until the cell stabilisers fix the cells pointwise. 

We now give a method for checking effectively if the subdivision is fine enough for the latter property to hold.
First we compute which vertices of the Bianchi fundamental polyhedron lie on the same $\Gamma$-orbit. This can be deduced from the operation formula stated in the beginning of section \ref{geometricRigidity} and has been implemented in \cite{BianchiGP}.
Then we perform a check to make sure that no edge of the Bianchi fundamental polyhedron can be sent onto itself reversing its orientation.
The latter is only possible when the two endpoints of the edge are identified by some element of~$\Gamma$.
To avoid this circumstance, we subdivide barycentrically all the edges the two endpoints of which are identified by some element of~$\Gamma.$

In order to establish an analogous criterion for 2-cells, let us make use of the fact that real hyperbolic space is non-positively curved -- it has the CAT(0) property \cite{BridsonHaefliger}.

\begin{Lem}
Let $\sigma$ be a polygon, and $G$ be a group of isometries of an ambient {\rm CAT(0)} space.
Suppose that among the vertices of $\sigma$,  there are at least three which are the unique representatives of their respective $G$-orbit.
Then the stabiliser of $\sigma$ in $G$ must fix $\sigma$ pointwise. 
\end{Lem}
\begin{proof}
Consider an element $g$ of the stabiliser of $\sigma$. The isometry $g$ must  preserve the set of the vertices of $\sigma$ up to a permutation. 
Furthermore, a vertex which is not $G$-equivalent to any other in this set, must be fixed by $g$. 
Under the hypothesis of our lemma, $g$ must hence fix three vertices of $\sigma$. 
As $g$ is a CAT(0) isometry, it must fix pointwise the whole triangle with these three vertices as corners. 
This triangle is contained in the polygon $\sigma$ and determines the isometric automorphisms of $\sigma$. Thus $g$ must fix $\sigma$ pointwise.
\end{proof}
Hence the check on our cell structure consists of making sure that each 2-cell has at least three vertices which are unique as representative of their respective $\Gamma$-orbit, among the vertices of the 2-cell.
Again, we can do this because we already have computed the $\Gamma$-equivalence classes of vertices.

The guarantee that all cells are fixed pointwise, allows us to obtain the stabilisers of the higher dimensional cells simply by taking the intersection of their vertex stabilisers.
Even more, in order to check the equivalence of two cells $\sigma$ and $\sigma'$, we only need to intersect the sets of elements of $\Gamma$ which identify the vertices of $\sigma$ with the ones of $\sigma'$.
The following lemma applies to any $\Gamma$-cell complex in hyperbolic space, and to all the cells in the refined cell complex.

\begin{Lem} \label{pointwiseStabilizers}
The stabilisers in $\mathrm{PSL}_2(\ringO)$ of pointwise-fixed 
\begin{itemize} 
\item edges in $\Hy$, are cyclic groups of orders one, two or three;
\item $2$--cells and $3$--cells in $\Hy$, are trivial.
\end{itemize}
\end{Lem}
\begin{proof}
As $\mathrm{SL}_2(\C)$ acts as orientation-preserving isometries on hyperbolic three-space, the stabiliser of a pointwise-fixed edge can only perform a rotation, with this edge lying on the rotation axis. This is possible because the edges in $\Hy$ are geodesic segments.  The group of rotations around one given axis must be abelian; and it is easy to see that it cannot be of Klein four-group type.
Thus among the subgroups of $\mathrm{PSL}_2(\ringO)$ which fix points in $\Hy$ --- their types are listed in lemma~\ref{finiteSubgroups} --- the only non-trivial types of groups which can fix edges pointwise are $\Z/2$ and $\Z/3$.
\\
In a pointwise-fixed 2-cell or 3-cell, we can choose two non-aligned pointwise-fixed edges, a rotation around one of which only fixes the other edge pointwise if it is the trivial rotation.
\end{proof}

A completely different cell complex for extensions of the Bianchi groups has been obtained by Yasaki~\cite{Yasaki}, who has implemented an algorithm of Gunnells~\cite{Gunnells} to compute the perfect forms modulo the action of $\mathrm{GL_2}(\mathcal{O}_{-m})$, giving the facets of the Vorono\"{\i} \thinspace polyhedron arising from a construction of Ash~\cite{Ash}.

\section{Rigidity of the action on the refined cell complex} \label{Rigidity of the action on the refined cell complex}

\begin{Lem} \label{twoEdges}
Each $\Gamma_v$--axis contains two edges of the refined cell complex that are adjacent to $v$.
\end{Lem}
\begin{proof}
Let $l$ be a $\Gamma_v$--axis for which this is not the case. 
Then $l$ passes through the interior of a $2$- or $3$-cell $\sigma$ adjacent to $v$.
Let $\gamma$ be a non-trivial element of $\Gamma$ fixing $l$.
As the $\Gamma$-action preserves our cell structure, $\gamma$ must send $\sigma$ to another cell of its dimension.
Since $\gamma$ fixes the points in the non-empty intersection of $l$ with the interior of $\sigma$, and since the interior of $\sigma$ intersects trivially with the interior of any other cell, $\gamma$ must fix $\sigma$.
Hence $\gamma$ is in the stabiliser of $\sigma$, 
and trivial by lemma \ref{pointwiseStabilizers}.
This contradicts the assumption that $\gamma$ has been chosen non-trivially. 
\end{proof}

\begin{Lem} \label{onAxis}
Let $e$ be an edge fixed pointwise by a non-trivially acting element $\gamma \in \Gamma$. Let $v$ be a vertex adjacent to $e$. Then $e$ lies on a \mbox{$\Gamma_v$--axis.}
\end{Lem}
\begin{proof} 
By corollary \ref{corollaryTo1.4}, $\gamma$ must perform a rotation around an axis passing through~$v$ and all the points of $e$. This is the $\Gamma_v$--axis containing $e$.
\end{proof}

\begin{lemma} \label{sublemma}
Let $\gamma \in \Gamma_v$ be a non-trivially acting element, stabilising a vertex $v \in \Hy$. Then the following two assertions are equivalent.
\begin{enumerate}
 \item[(i)] Modulo the action of $\Gamma_v$, there is just one edge adjacent to $v$ on the $\Gamma_v$--axis stabilised by~$\langle \gamma \rangle$.
 \item[(ii)] The normaliser of $\langle \gamma \rangle$ in $\Gamma_v$ contains an element carrying out an isometry of order $2$ and which is not contained in~$\langle \gamma \rangle$.
\end{enumerate}
\end{lemma}
\begin{proof}
Denote by $l$ the $\Gamma_v$--axis pointwise stabilised by $\langle \gamma \rangle$. By lemma~\ref{twoEdges}, adjacent to $v$ there are two edges of the refined cell complex on $l$.
\\
If they are identified by an element $\beta \in \Gamma_v$, the isometry $\beta$ must send $l$ to itself. Then \begin{itemize}
\item corollary~\ref{normaliser in the stabiliser} tells us that $\beta$ is in the normaliser of $\langle \gamma \rangle$ in $\Gamma_v$.
\item Denoting the two identified edges by $(a,v)$ and $(b,v)$, it follows from $\beta(a,v) = (b,v)$ that $(a,v) = \beta(b,v)$ and vice versa. Hence the isometry carried out by $\beta$ cannot be of order~$1$ or~$3$. The elements of~$\Gamma_v$ all carry out isometries of orders~$1$, $2$ or~$3$, 
\end{itemize}
so (i) implies (ii).
\\
Conversely, assume that $\beta$ is an element of the normaliser of $\langle \gamma \rangle$ in $\Gamma_v$, carrying out an isometry of order $2$ and not contained in~$\langle \gamma \rangle$.
Then by corollary~\ref{normaliser in the stabiliser}, it sends $l$ to itself.
If this happens by the identity on $l$, then $l$ is the rotation axis of $\beta$ and by lemma~\ref{axes-cyclicSubgroupsBijection}, $\beta$ is in~$\langle \gamma \rangle$, contrary to our assumption. Hence $\beta$ sends the two edges on $l$ to one another, and assertion (i) follows. 
\end{proof}

The above study 
provides all the tools to prove the following lemma, which is useful in order to obtain theorem~\ref{HomologicalInvarianceUnderHomeomorphisms}.
We use the notations of lemma~\ref{finiteSubgroups} for the occurring types of finite groups.

\vbox{
\begin{lemma} \label{geometricRigiditytheorem}
Let $v$ be a non-singular vertex in the refined cell complex. Then the number~$\bf n$ of orbits of edges in the refined cell complex adjacent to $v$, with stabiliser in ${\rm PSL}_2(\ringO_{-m})$ isomorphic to~$\Z/ \ell$,  is given as follows for $\ell = 2$ and $\ell = 3$.
$$ \begin{array}{c|cccccc}
{\rm Isomorphism}\medspace { \rm type}\medspace {\rm of } \medspace {\rm the  }\medspace {\rm vertex }\medspace {\rm stabiliser} & \{1\} & \Z/2 & \Z/3 & \Kleinfourgroup & \Sthree & \Afour \\ 
\hline &&&&&& \scriptsize \\
{\bf n} \medspace \mathrm{ for } \medspace \ell = 2 & 0 & 2 & 0 & 3 & 2 & 1 \\ 
{\bf n} \medspace \mathrm{ for } \medspace \ell = 3 & 0 & 0 & 2 & 0 & 1 & 2.
\end{array} $$ \normalsize
\end{lemma}
}

 \begin{proof}
Due to lemma \ref{onAxis}, any edge with the requested properties must lie on some 
\mbox{$\Gamma_v$--axis.} 
Thus the cases where $n = 0$ follow directly from 
lemma \ref{axes-cyclicSubgroupsBijection}.
It remains to distinguish the following cases.
\begin{itemize}
\item 
Let $\ell = 2$ and $\Gamma_v \cong \Afour$.
\\
There is only one conjugacy class of order-2-elements in $\Afour$, 
so by theorem~\ref{actionOnAxes}, there is just one $\Gamma_v$-orbit of \mbox{$\Gamma_v$--axes} such that an element of order 2 fixes this axis pointwise.
The normaliser of any element of order $2$ in $\Afour$ is the Sylow 2-subgroup $\Kleinfourgroup$.
So by lemma~\ref{sublemma}, just one $\Z/2$--stabilised edge adjacent to $v$ is in the quotient by $\Gamma_v$.
\item 
Let $\ell = 2$ and $\Gamma_v \cong \Kleinfourgroup$.
\\
There are exactly three conjugacy classes of order-2-elements in $\Kleinfourgroup$, 
so by theorem~\ref{actionOnAxes}, there are three $\Gamma_v$-orbits of \mbox{$\Gamma_v$--axes.} 
The whole group $\Kleinfourgroup$ normalises any of its elements.
So by lemma~\ref{sublemma}, there are exactly three representatives of non-trivially stabilised edges adjacent to~$v$, one on each \mbox{$\Gamma_v$--axis.} Their stabilisers are the order-2-subgroups of~$\Gamma_v$, as we see from lemma~\ref{axes-cyclicSubgroupsBijection}.
\item 
Let $\ell = 2$ and $\Gamma_v \cong \Sthree$.
 \\
There is only one conjugacy class of order-2-elements in $\Sthree$, 
so by theorem~\ref{actionOnAxes}, there is just one $\Gamma_v$-orbit of \mbox{$\Gamma_v$--axes} such that an element of order 2 fixes this axis pointwise.
The normaliser of an order-2-subgroup $\langle \gamma \rangle$ in $\Sthree$ is $\langle \gamma \rangle$ itself.
So by lemma~\ref{sublemma}, the two edges on this axis lie on different orbits.
\item Let $\Gamma_v$ be a non-trivial cyclic group. \\
As we see from corollary~\ref{corollaryTo1.4}, any vertex with stabiliser a cyclic group lies on a single rotation axis, around which its stabiliser performs rotations.
From lemma~\ref{sublemma}, we see that the two edges on this axis lie on different orbits.
\item
Let $\ell = 3$ and $\Gamma_v \cong \Sthree$. 
\\
There is only one cyclic subgroup of order $3$ in $\Sthree$.
Its normaliser in $\Sthree$ is the full group~$\Sthree$.
From lemma~\ref{sublemma}, we see that there is just one relevant edge in the quotient by~$\Gamma_v$.
\item 
Let $\ell = 3$ and $\Gamma_v \cong \Afour$. 
\\
There are four cyclic subgroups of order $3$ in $\Afour$.
They are all conjugate, so just one \mbox{$\Gamma_v$--axis} is in the quotient by $\Gamma_v$.
The normaliser of $\Z/3$ in $\Afour$ is only $\Z/3$ itself, so by lemma~\ref{sublemma} we obtain two edges on different orbits.
\end{itemize}
\end{proof}

The proof of the following corollary is included in the above proof.
\begin{corollary} ${}$
\begin{itemize} 
\item In the case $\Gamma_v \cong \Kleinfourgroup$ of lemma~\emph{\ref{geometricRigiditytheorem}}, the three stabilisers of edge representatives adjacent to $v$ which are not trivial, are precisely the three order-$2$-subgroups of~$\Gamma_v$. 
\item Consider the cases of lemma~\emph{\ref{geometricRigiditytheorem}} where $\Gamma_v$ is a non-trivial cyclic group. Then the two edges adjacent to $v$, which have a non-trivial stabiliser, have the same stabiliser as~$v$.
\end{itemize}
\end{corollary}

By observing the pre-images of the projection from $\mathrm{SL_2}(\ringO_{-m})$ to $\mathrm{PSL_2}(\ringO_{-m})$, we further obtain the following.

\begin{corollary} 
Let $v$ be any vertex in the refined cell complex. Then the number $n$ of orbits of edges in the refined cell complex adjacent to $v$, with stabiliser in  $\Gamma :=\mathrm{SL_2}(\ringO_{-m})$ isomorphic to $\Z/ 2\ell$,  is given by the table of lemma~{\rm \ref{geometricRigiditytheorem}}, if we replace the stabiliser isomorphism types $\{1\}$, $\Z/2$, $\Z/3$, $\Kleinfourgroup$, $\Sthree$ and $\Afour$ by their pre-images, which are respectively:  $\Z/2$, $\Z/4$, $\Z/6$, the $8$-elements quaternion group, the $12$-elements binary dihedral group and the binary tetrahedral group.
\end{corollary}

\section{The Fl\"oge cellular complex} \label{Floege cellular complex}

In order to obtain a cell complex with compact quotient space, we proceed in the following way due to Fl\"oge \cite{Floege}.
The boundary of $\Hy$ is the Riemann sphere~$\partial \Hy$, which, as a topological space, is made up of the complex plane $\C$ compactified with the cusp $\infty$.
The totally geodesic surfaces in $\Hy$ are the Euclidean vertical planes (we define \emph{vertical} as orthogonal to the complex plane) and the Euclidean hemispheres centred on the complex plane.
The action of the Bianchi groups extends continuously to the boundary \mbox{$\partial \Hy$}. 
The cellular closure of the refined cell complex in $\Hy \cup \partial \Hy $ consists of \mbox{$ \Hy$} and the set of cusps \mbox{$\left(\rationals (\sqrt{-m}) \cup \{\infty\}\right) \subset \left(\C \cup \{\infty\}\right) \cong \partial \Hy$.} 
The $\mathrm{SL_2}(\ringO_{-m})$--orbit of a cusp $\frac{\lambda}{\mu}$ in $\left(\rationals (\sqrt{-m}) \cup \{\infty\}\right)$ corresponds to the ideal class $[(\lambda, \mu)]$ of $\ringO_{-m}$. It is well-known that this does not depend on the choice of the representative $\frac{\lambda}{\mu}$.
We extend the refined cell complex to a cell complex $\widetilde{X}$ by joining to it, in the case that $\ringO_{-m}$ is not a principal ideal domain, the $\mathrm{SL_2}(\ringO_{-m})$--orbits  of the cusps $\frac{\lambda}{\mu}$ for which the ideal $(\lambda, \mu)$ is not principal. 
At these cusps, we equip $\widetilde{X}$ with the ``horoball topology'' described in~\cite{Floege}. This simply means that the set of cusps, which is discrete in \mbox{$\partial \Hy$}, is located at the hyperbolic extremities of $\widetilde{X}$ : No neighbourhood of a cusp, except the whole of $\widetilde{X}$, contains any other cusp.

We retract $\widetilde{X}$ in the following, $\mathrm{SL_2}(\ringO_{-m})$--equivariant, way.
On the Bianchi fundamental polyhedron, the retraction is given by the vertical projection (away from the cusp~$\infty$) onto its facets which are closed in $\Hy \cup \partial \Hy$. The latter are the facets which do not touch the cusp $\infty$, and are the bottom facets with respect to our vertical direction. The retraction is continued on $\Hy$ by the group action. It is proven in \cite{FloegePhD} that this retraction is continuous.
We call the retract of~$\widetilde{X}$ the \emph{Fl\"oge cellular complex} and denote it by $X$.
So in the principal ideal domain cases, $X$ is a retract of the refined cell complex, obtained by contracting the Bianchi fundamental polyhedron onto its cells which do not touch the boundary of $\Hy$.  
In \cite{RahmFuchs}, it is checked that the Fl\"oge cellular complex is contractible. 

A cell complex constructed by Mendoza \cite{Mendoza} that coincides in the principal ideal domain cases with the Fl\"oge cellular complex has been implemented by Vogtmann \cite{Vogtmann}.

\section{Equivariant Euler characteristic} \label{Equivariant Euler characteristic}
We use the Euler characteristic to check the geometry of the quotient $\Gamma \backslash X$.
Recall the following definitions and proposition.
\begin{df}[Euler characteristic]
Suppose $\Gamma'$ is a torsion-free group. Then we define its Euler characteristic as 
$\chi(\Gamma')=\sum_i (-1)^i \dim \Homol_i(\Gamma';\rationals).$
Suppose further that $\Gamma'$ is a torsion-free subgroup of finite index in a group $\Gamma$.
Then we define the {\em Euler characteristic} of $\Gamma$  as 
$\chi(\Gamma)= \frac{\chi(\Gamma')}{[\Gamma : \Gamma']}.$
\end{df}
The latter formula is well-defined because of \cite{Brown}*{IX.6.3}. If in the first formula we drop the condition that $\Gamma'$ is torsion-free, we obtain the \emph{naive} Euler characteristic.

\begin{df}[Equivariant Euler characteristic]
Suppose $X$ is a $\Gamma$-complex such that
\begin{itemize}
\item
every isotropy group $\Gamma_\sigma$ is of finite homological type;
\item
$X$ has only finitely many cells mod $\Gamma$.
\end{itemize}
Then we define the $\Gamma$-{\em equivariant Euler characteristic} of $X$ as
$ \chi_\Gamma(X) := \sum\limits_\sigma (-1)^{\mathrm{dim}\sigma} \chi(\Gamma_\sigma),$
where~$\sigma$ runs over the orbit representatives of cells of $X$.
\end{df}
\begin{prop}[\cite{Brown}*{IX.7.3 e'}]
Suppose $X$ is a $\Gamma$-complex such that $ \chi_\Gamma(X)$ is defined.
If $\Gamma$ is virtually torsion-free, then $\Gamma$ is of finite homological type and $ \chi(\Gamma) =  \chi_\Gamma(X).$
\end{prop}
Let now $\Gamma$ be $\text{PSL}_2\bigl(\mathcal{O}_{-m}\bigr)$. Then the above proposition applies to $X$ taken to be Fl\"oge's (or still, Mendoza's) $\Gamma$-equivariant deformation retract of $\Hy$.
 Using $\chi(\Gamma_\sigma) = \frac{1}{\mathrm{card}(\Gamma_\sigma)}$ for $\Gamma_\sigma$ finite,
 the fact that the singular points have stabiliser $\Z^2$,
and the torsion-free Euler characteristic 
$$\chi(\Z^2) = \sum\limits_i (-1)^i \mathrm{rank}_\Z (\Homol_i \Z^2) = 1-2+1 = 0,$$ we get the formula
$$ \chi(\Gamma) = 
\sum\limits_\sigma (-1)^{\mathrm{dim}\sigma} \frac{1}{\mathrm{card}(\Gamma_\sigma)},$$
where $\sigma$ runs over the orbit representatives of cells of $X$ with finite stabilisers.
\begin{prop} \label{Euler_characteristic_vanishes}
The Euler characteristic $\chi(\Gamma)$ vanishes for the Bianchi groups. 
\end{prop}
This is a well-known fact, see for instance \cite{RahmFuchs} for a proof.  We obtain a ``mass formula''
$$
0  =  \sum\limits_\sigma (-1)^{\mathrm{dim}\sigma} \frac{1}{\mathrm{card}(\Gamma_\sigma)},
$$
which allows us to check the topology of the computed quotient space. 
For example, in the case $m = 427$, the mass formula takes the expression
$$ 184
+\frac{12}{2}
+\frac{18}{3}
+\frac{4}{6}
 -\left(
441
+\frac{16}{2}
+\frac{20}{3}
\right) +259 = 0, $$
which comes from 184 trivially stabilised vertices, 12 vertices with stabiliser of order two, 18 vertices with stabiliser of order three, 4 vertices with stabiliser of type $\Sthree$, 441~trivially stabilised edges, 16 edges with stabiliser of order two, 20 edges with stabiliser of order three, and 259 two-cells in the quotient cell complex.
Tables with the expression in other cases, including all cases of class number 2, are given in \cite{RahmThesis}.
A further check which has been carried out on the results of \cite{BianchiGP} is the vanishing of the naive Euler characteristic, which is proven in \cite{Vogtmann}.

\section{The maps induced on cohomology by finite subgroup inclusions} \label{The maps induced on cohomology by finite subgroup inclusions}

Concerning the possible finite stabiliser groups of vertices in hyperbolic space, we can determine their integral homology by the methods described in \cite{AdemMilgram}. This has been done in \cite{SchwermerVogtmann}, and we obtain their homology with $\Z/ \ell$--coefficients by the universal coefficient theorem. Then it only remains to correct a minor typographical error for $\Afour$ in order to obtain the following lemma.
\begin{lemma}[Schwermer/Vogtmann] \label{finiteSubgroupsHomology}
The homology with trivial $\Z$-- respectively $\Z/\ell$--coefficients, for $\ell = 2$ or $3$, of the finite subgroups of PSL$_2(\ringO_{-m})$ listed in lemma~{\rm \ref{finiteSubgroups}} is
\tiny
\begin{alignat*}6
&
\Homol_q(\Z/n; \thinspace  \thinspace \Z)
&
 \cong
 & 
  \begin{cases}
   \Z, & q = 0, \\
   \Z/n, & q \medspace \mathrm{ odd}, \\
   0, & q \medspace \mathrm{ even}, \medspace q > 0;
   \end{cases}
&
\quad
& 
 \Homol_q(\Z/n; \thinspace \Z/n)
&\cong
&\thinspace \Z/n, \medspace \mathrm{for} \medspace n, \thinspace q \in \N \cup \{0\};
&\quad
  &
&
&
\\
& \Homol_q({\mathcal{D}}_2; \thinspace \Z)
 &
 \cong 
 &
\begin{cases}
 \Z, & q = 0, \\
 (\Z/2)^\frac{q+3}{2}, & q \medspace \mathrm{ odd}, \\
 (\Z/2)^\frac{q}{2}, & q \medspace \mathrm{ even}, \medspace q > 0;
 \end{cases}
&\quad&
\Homol_q({\mathcal{D}}_2; \thinspace \Z/2) &
\cong
& (\Z/2)^{q+1}
&\quad&
   \Homol_q({\mathcal{D}}_2; \thinspace \Z/3) &=& 0, q \geq 1;\\
&  \Homol_q({\mathcal{S}}_3; \thinspace \Z) &\cong& 
\begin{cases}
 \Z, & q = 0, \\
 \Z/2, & q \equiv 1 \mod 4, \\
 0, & q \equiv 2 \mod 4, \\
 \Z/6, & q \equiv 3 \mod 4, \\
 0, & q \equiv 0 \mod 4, \medspace q > 0;
 \end{cases} 
 & \quad &
  \Homol_q({\mathcal{S}}_3; \thinspace \Z/3) &\cong& 
\begin{cases}
 \Z/3, & q = 0, \\
 0, & q \equiv 1 \mod 4, \\
 0, & q \equiv 2 \mod 4, \\
 \Z/3, & q \equiv 3 \mod 4, \\
 \Z/3, & q \equiv 0 \mod 4, \medspace q > 0;
 \end{cases} 
 &\quad & \Homol_q({\mathcal{S}}_3; \thinspace \Z/2) &\cong& \Z/2, q \in \N \cup \{0\};\\
&  \Homol_q({\mathcal{A}}_4; \thinspace \Z) &\cong& 
\begin{cases}
 \Z, & q = 0, \\
 (\Z/2)^k \oplus \Z/3, & q = 6k+1, \\
 (\Z/2)^k \oplus \Z/2, & q = 6k+2, \\
 (\Z/2)^k \oplus \Z/6, & q = 6k+3, \\
 (\Z/2)^k , & q = 6k+4, \\
 (\Z/2)^k \oplus \Z/2 \oplus \Z/6, & q = 6k+5, \\
 (\Z/2)^{k+1}, & q = 6k+6. \\
 \end{cases} 
& \qquad &
\Homol_q({\mathcal{A}}_4; \thinspace \Z/2) &\cong&
 \begin{cases}
\Z/2, & q = 0, \\
(\Z/2)^{2k}, & q = 6k+1, \\
(\Z/2)^{2k+1}, & q = 6k+2, \\
(\Z/2)^{2k+2}, & q = 6k+3, \\
(\Z/2)^{2k+1} , & q = 6k+4, \\
(\Z/2)^{2k+2}, & q = 6k+5, \\
(\Z/2)^{2k+3}, & q = 6k+6. \\
\end{cases} 
 &\quad & \Homol_q({\mathcal{A}}_4; \thinspace \Z/3) &\cong &\Z/3, q \in \N \cup \{0\}.
\end{alignat*}
\normalsize
\end{lemma}

With the Lyndon/Hochschild/Serre spectral sequence, we can further establish the following.

\begin{Lem}[Schwermer/Vogtmann \cite{SchwermerVogtmann}] \label{inducedMaps} 
Let $M$ be $\Z$ or $\Z/2$. Consider group homology with trivial $M$-coefficients. Then the following holds.
\begin{itemize}
\item Any inclusion $\Z/2 \to \Sthree$ induces an injection on homology.
\item An inclusion $\Z/3 \to \Sthree$ induces an injection on homology in degrees congruent to $3$ or $0 \mod 4$, and is otherwise zero.
\item Any inclusion $\Z/2 \to \Kleinfourgroup$ induces an injection on homology in all degrees.
\item An inclusion $\Z/3 \to \Afour$ induces injections on homology in all degrees.
\item An inclusion $\Z/2 \to \Afour$ induces injections on homology in degrees greater than $1$, and is zero on $\Homol_1$. 
\end{itemize}
\end{Lem}

Schwermer and Vogtmann prove this for $M = \Z$. \\
We will make use of the following sub-lemmata to prove it for~$M = \Z/2$. As the only automorphism of~$\Z/2$ is the identity, \mbox{$\Z/2$--coefficients} are always trivial coefficients.

\begin{sublem} \label{inclusionsIntoD2}
Lemma \emph{\ref{inducedMaps}} holds in the case $M = \Z/2$ for the inclusions into $\Kleinfourgroup$.
\end{sublem}
\begin{proof}
Consider the Lyndon/Hochschild/Serre spectral sequence with $\Z/2$-coefficients of the trivial extension $1 \to \Z/2 \to \Kleinfourgroup \to  \Z/2 \to 1 $. It takes the form
$$ E^2_{p,q} = \Homol_p \left(\Z/2; \Homol_q(\Z/2;\Z/2)\right) \Rightarrow \Homol_{p+q}(\Kleinfourgroup; \Z/2).$$
As $\Homol_q(\Z/2;\Z/2) \cong \Z/2$ for all $q \in \N \cup \{0\}$, we obtain $E^2_{p,q} = \Z/2$ for all $p, q \in \N \cup \{0\}$. 
As we know from lemma~\ref{finiteSubgroupsHomology} that 
$\Homol_q(\Kleinfourgroup;\Z/2) \cong (\Z/2)^{q+1}$, all the differentials must be zero and $E^2 = E^\infty$. Hence we obtain the claimed injections on homology. 
\end{proof}

\begin{sublem} \label{inclusionsIntoS3}
Lemma \emph{\ref{inducedMaps}} holds in the case $M = \Z/2$ for the inclusions into $\Sthree$.
\end{sublem}
\begin{proof}
Consider the Lyndon/Hochschild/Serre spectral sequence with $\Z/2$-coefficients of the non-trivial extension $1 \to \Z/3 \to \Sthree \to \Z/2 \to 1$.
It takes the form
$$ E^2_{p,q} = \Homol_p \left(\Z/2; \Homol_q(\Z/3; \Z/2) \right) \Rightarrow \Homol_{p+q}(\Sthree; \Z/2).$$
As $\Homol_q(\Z/3; \Z/2) = 0$ for $q > 0$, the $E^2$-page is concentrated in the row $q = 0$ and equals the $E^\infty$-page. Thus we have isomorphisms 
$ \Homol_p \left(\Z/2; \Homol_0(\Z/3; \Z/2) \right) \cong \Homol_{p}(\Sthree; \Z/2),$ from which we obtain the claimed morphisms on homology.
\end{proof}

Let $t$ be a generator of $\Z/3$. Let $F$ be the periodic resolution of $\Z$ over $\Z[\Z/3]$ given by 
$$\xymatrix{ \ldots \ar[r]^{t-1 \qquad}  & \Z[\Z/3] \ar[rr]^{t^2+t+1} & & \Z[\Z/3] \ar[r]^{t-1}  & \Z[\Z/3] \ar[rrr]^{\rm augmentation} & & & \Z.}$$

\begin{Lem} \label{moduleOverRingOfCharacteristic2}
Let $A$ be an abelian group consisting only of elements of order $2$ and the identity element.
Then $F \otimes_{\Z[\Z/3]} A$ is exact in degrees greater than zero, regardless of the $\Z[\Z/3]$-module structure attributed to $A$.
\end{Lem}
\begin{proof}
We will show the two equations \begin{center} $\im (t^2+t+1) = \ker (t-1)$ and 
\mbox{$\ker (t^2+t+1) = \im (t-1)$.} \end{center}
As $t^3 = 1$, the equation $(t^2+t+1)(t-1) = 0$ holds, and yields the inclusions \begin{center}
$\im (t^2+t+1) \subset \ker (t-1)$  and $\ker (t^2+t+1) \supset \im (t-1)$. \end{center}
Now we want to show that $\im (t^2+t+1) \supset \ker (t-1)$.
Let $v \in \ker (t-1)$. Then $(t-1)\cdot v = 0$, or equivalently, $t \cdot v = v$. 
We apply this three times to obtain $(t^2+t+1)\cdot v = 3v$. As $2v = 0$ in~$A$, we have $(t^2+t+1)\cdot v = v$ and hence the claimed inclusion.
\\
It remains to show that $\ker (t^2+t+1) \subset \im (t-1)$. Let $v \in \ker (t^2+t+1)$.
\\
We will see that $t \cdot v$ maps to  $v$ under multiplication by $t-1$.
Namely, \\
$(t-1)\cdot (t \cdot v) = t^2 \cdot v -t \cdot v = -2t \cdot v -v$ because $t^2 \cdot v + t \cdot v +v = 0$. Using that multiplication by $2$ is zero on $A$, we obtain the image $v$ and hence the last inclusion.
\end{proof}

\begin{sublem}\label{LHS}
The $E^2$-page of the Lyndon/Hochschild/Serre spectral sequence with $\Z/2$--coefficients for the extension 
$1 \to \Kleinfourgroup \to \Afour \to \Z/3 \to 1 $ is concentrated in the column $p = 0$.
\end{sublem}
\begin{proof}
The $E^2$-page of the Lyndon/Hochschild/Serre spectral sequence with $\Z/2$--coefficients
is given by $E^2_{p,q} = \Homol_p \left(\Z/3; \Homol_q(\Kleinfourgroup; \Z/2)\right)$. 
\\
The action of $\Z/3$ on \mbox{$\Homol_q(\Kleinfourgroup; \Z/2) \cong (\Z/2)^{q+1}$} is determined by the non-trivial conjugation action of $\Z/3$ on $\Kleinfourgroup$. 
But applying lemma~\ref{moduleOverRingOfCharacteristic2}, we obtain 
$\Homol_p \left(\Z/3; (\Z/2)^{q+1} \right) = 0$ for $p > 0$ and any action of $\Z/3$ on  $(\Z/2)^{q+1}$.
\end{proof}

\mbox{We can now do the last two cases to prove lemma~\ref{inducedMaps} in the case $M = \Z/2$. }

\begin{sublem}
Lemma \emph{\ref{inducedMaps}} holds in the case $M = \Z/2$ for the inclusions into $\Afour$.
\end{sublem}
\begin{proof} ${}$

\begin{itemize}
\item For an inclusion $\Z/3 \to \Afour$, this follows from the fact that 
\mbox{$\Homol_p(\Z/3; \Z/2) = 0$ for all $p > 0$.}
\item For an inclusion $\Z/2 \to \Afour$, factorise by an inclusion $\Z/2 \to \Kleinfourgroup$.
Lemma~\ref{LHS} gives an isomorphism $\Homol_q(\Afour; \Z/2) \cong \Homol_0 \left(\Z/3; \Homol_q(\Kleinfourgroup; \Z/2)\right)$ for $q \in \N \cup  \{0\}.$ 
From lemma~\ref{inclusionsIntoD2} and the low terms of our resolution for $\Afour$ given in the appendix, we deduce that the induced map on homology is injective whenever $\Homol_q(\Afour; \Z/2)$ is non-zero. 
From lemma~\ref{finiteSubgroupsHomology}, we see that this is the case for all $q \in \N \cup \{0\}$ except for $q = 1$, where we obtain the zero map.
\end{itemize}
\end{proof}

\section{ The equivariant spectral sequence to group homology} \label{SpecSeq}
Let $\Gamma$ be a Bianchi group. We will use the Fl\"oge cellular complex $X$ to compute the group homology of $\Gamma$ with trivial $\Z$--coefficients, as defined in \cite{Brown}. We proceed following \cite{Brown}*{VII} and \cite{SchwermerVogtmann}.
Let us consider the homology $\Homol_*(\Gamma; C_\bullet(X))$ of $\Gamma$ with coefficients in the cellular chain complex~$C_\bullet(X)$ associated to $X$, and call it the $\Gamma$-equivariant homology of $X$.
As $X$ is contractible, the map $X \to pt.$ to the single point $pt.$ induces an isomorphism
$$\Homol_*(\Gamma; \thinspace C_\bullet(X)) \to \Homol_*(\Gamma; \thinspace C_\bullet(pt.)) \cong \Homol_*(\Gamma;\thinspace \Z).$$
Denote by $X^p$ the set of $p$-cells of $X$, and make use of that the stabiliser $\Gamma_\sigma$ in $\Gamma$ of any $p$-cell $\sigma$ of $X$ fixes $\sigma$ pointwise. Then from 
$$ C_p(X) = \bigoplus\limits_{\sigma \in X^p} \Z 
\cong \bigoplus\limits_{\sigma \thinspace\in \thinspace _\Gamma \backslash X^p} {\rm Ind}^\Gamma_{\Gamma_\sigma} \Z,$$
Shapiro's lemma yields
$$\Homol_{q}(\Gamma; \thinspace C_p(X)) 
\cong \bigoplus_{\sigma\thinspace\in\thinspace _\Gamma\backslash X^p}\Homol_q(\Gamma_\sigma; \thinspace \Z);$$
and the equivariant Leray/Serre spectral sequence takes the form 
$$
E^1_{p,q}=\bigoplus_{\sigma\thinspace\in\thinspace _\Gamma\backslash X^p}\Homol_q(\Gamma_\sigma; \thinspace \Z)\implies \Homol_{p+q}(\Gamma; \thinspace C_\bullet(X)),
$$
converging to the $\Gamma$-equivariant homology of $X$, which is, as we have already seen, isomorphic to 
$\Homol_{p+q}(\Gamma; \thinspace \Z)$ with the trivial action on the coefficients $\Z$.

We shall also make extensive use of the description given in \cite{SchwermerVogtmann}, of the $d^1$-differential in this spectral sequence. 
The technical difference to the cases of trivial class group, treated in \cite{SchwermerVogtmann}, 
is that the stabilisers of the singular points are free abelian groups of rank two. 
In particular, the $\Gamma$-action on our complex $X^\bullet$ is not a \textit{proper action} (in the sense that all stabilisers are finite). As a consequence, the resulting spectral sequence does not degenerate on the $E^2$-level like it does in Schwermer and Vogtmann's cases.
It is explained in \cite{RahmFuchs} how to handle the non-trivial $d^2$-differentials in this spectral sequence.

\subsection{The differentials} \label{the differentials}
Let us now describe how to compute explicitly the $d^1$-differentials,
 making use of the knowledge from lemma \ref{finiteSubgroups} about the isomorphism types of the stabilisers, and lemma \ref{inducedMaps} about their inclusions. 
The bottom row of the $E^1$-term, more precisely the chain complex given by the $E^1_{p,0}$-modules and the  $d^1_{p,0}$-maps, is equivalent to the $\Z$-chain complex giving the homology of the quotient space of our cell complex by the $\Gamma$-action.

From lemma \ref{pointwiseStabilizers}, we see that for $q>0$, the $E^1_{p,q}$-terms are concentrated in the two columns $p = 0$ and $p = 1$. So for $q>0$,  we only need to compute the differentials
$$
\bigoplus_{\sigma \thinspace\in \thinspace _\Gamma\backslash X^0} \Homol_q (\Gamma_\sigma; \Z)
\xleftarrow{\ d^1_{1,q}\ }
\bigoplus_{\sigma \thinspace\in \thinspace _\Gamma\backslash X^1} \Homol_q(\Gamma_\sigma; \Z).
$$

These differentials arise from the following cell stabiliser inclusions.
For any edge in $_\Gamma\backslash X^1$, we have, because it is fixed pointwise, an inclusion $\iota$ of its stabiliser into the stabiliser of its origin vertex. 
Choose any matrix $g$ which sends the origin vertex of this edge to its vertex representative in $_\Gamma\backslash X^0$.
The cell stabiliser inclusion associated to the origin of our edge is the composition of the conjugation by $g$ after the inclusion $\iota$.
Up to inner automorphisms of the origin vertex stabiliser, this conjugation map does not depend on the choice of $g$, because $g$ is determined up to multiplication with elements of the origin vertex stabiliser.
The cell stabiliser inclusion associated to the end of an edge is obtained analogously.
We see in \cite{Brown}*{VII.8} that these cell stabiliser inclusions induce the differential  $d^1$ of the equivariant spectral sequence.

As we see from lemma~\ref{pointwiseStabilizers}, 
the inclusions determined by lemma~\ref{inducedMaps} are the only non-trivial inclusions which occur in our $\mathrm{PSL}_2(\ringO_{-m})$-cell complex. 
Hence we can decompose the $d^1_{1,q}$ differential in the associated equivariant spectral sequence, for $q > 0$, into a 2-primary and a 3-primary part.

\begin{df}
For an abelian group $A$, the \emph{$\ell$-primary part} is the subgroup consisting of all elements of $A$ of $\ell$-power order.
\end{df}

\subsection{The 3-primary part} \label{3primary}

Denote by $(d^1_{1,q})_{(\ell)}$ the $\ell$-primary part of our $d^1_{1,q}$ differential.
It suffices to compute $(d^1_{1,1})_{(3)}$ and $(d^1_{1,3})_{(3)}$ to get the 3-primary part of our $d^1_{1,q}$ differential, because of the following.

\begin{cor} 
The $3$-primary part $(d^1_{1,q})_{(3)}$ is of period $4$ in $q$; and is zero for $q > 0$ even. 
\end{cor}
\begin{proof} Lemmata \ref{finiteSubgroupsHomology} and \ref{inducedMaps}.
\end{proof}

An algorithm for the computation of the rank of the differential matrices $(d^1_{1,1})_{(3)}$ and  $(d^1_{1,3})_{(3)}$ has been given in \cite{RahmThesis}.

\subsection{The 2-primary part} \label{2primary}
As by lemma~\ref{pointwiseStabilizers}, there are only edges with finite cyclic stabilisers, we see that the $d^1_{1,q}$ differential is zero for $q > 0$ even.
Now for $q$ odd, we want to compute its 2-primary part.
Any group monomorphism from edge stabilisers of type $\Z/2$ to vertex stabilisers of type $\Z/2$ or $\Sthree$ induces the only possible isomorphism  $\Z/2 \to \Z/2$ on homology.
\\
By lemma~\ref{inducedMaps}, for $q=1$ the monomorphisms $\Z/2 \to \Afour$ induce zero maps. So, consider the case  $q>1$ odd. By lemma~\ref{geometricRigiditytheorem}, when we establish a matrix for $(d^1_{1,q})_{(2)}$, its block associated to a vertex orbit of stabiliser type $\Afour$ in the sense of subsection~\ref{the differentials}, has exactly one non-zero column.

\begin{Lem} \label{D2blocks}
Let $q$ be odd. Let $v$ be a vertex representative of stabiliser type $\Kleinfourgroup$.
 Then the block associated to it in a matrix for the $(d^1_{1,q})_{(2)}$ differential, has exactly three non-zero columns.
 There is a basis for $\left(\Homol_q(\Kleinfourgroup; \thinspace \Z) \right)_{(2)}$ such that these block columns are
 $(1,0,\ldots,0)^{\rm t}$, $(0,\ldots,0,1)^{\rm t}$ and $(1,\ldots,1)^{\rm t}$ (with $\rm t$ the transpose).
 The latter are linearly independent if and only if $q \geq 3$.
\end{Lem}
\begin{proof} By lemma~\ref{geometricRigiditytheorem},
 there are exactly three $\Gamma$-representatives of non-trivially stabilised edges adjacent to $v$.
Furthermore, the stabilisers of these three edges are precisely the three order-$2$-subgroups of the stabiliser of $v$. 
We apply the chain map computation of \cite{SchwermerVogtmann}, to each of these three subgroup inclusions,
 and obtain the claimed three block columns.
The length of these block columns is the $\Z/2$-rank of $\Homol_q(\Kleinfourgroup ; \thinspace \Z)$,
 namely $\frac{q+3}{2}$, so we easily see that these block columns are linearly independent if and only if $q \geq 3$.
\end{proof}

\begin{prop}
Let $q \geq 3$ odd. Then ${\rm rank} (d^1_{1,q})_{(2)} = {\rm rank} (d^1_{1,3})_{(2)}.$ 
\end{prop}
\begin{proof}
We see from lemma~\ref{inducedMaps} that all inclusions of $\Z/2$ into any vertex stabiliser induce injections on homology in all degrees $q \geq 3$.
As the groups $\Z/2$ and $\Sthree$ have their \mbox{$q$-th} integral homology group $\Z/2$ for all odd $q$, there is just one possibility for induced injections into it;
and hence the matrix block of $(d^1_{1,q})_{(2)}$ associated to vertex stabilisers of these types is the same for all odd $q$.
Now for vertex stabilisers of type $\Afour$, we know from lemma~\ref{geometricRigiditytheorem} that there is just one 2-torsion edge representative stabiliser inclusion into them.
Thus the associated matrix blocks  $(1,0,\ldots,0)^{\rm t}$ only grow in the number of their zeroes when $q$ grows, but this does not change the rank of $(d^1_{1,q})_{(2)}$.
Finally we see from lemma~\ref{D2blocks} that associated to vertex representative stabilisers of type $\Kleinfourgroup$, there are exactly three matrix sub-blocks, which are linearly independent for all $q \geq 3$.
\end{proof}


Denote by $(\tilde{E},\tilde{d})$ the same equivariant spectral sequence, but now with $\Z/2$-coefficients.
\begin{Lem} \label{mod2ranks}
Let $q \geq 3$ odd. Then  the rank of $d^1_{1,q} \otimes \Z/2$  equals the ranks of the 
\mbox{differentials $\tilde{d}^1_{1,q}$} and $\tilde{d}^1_{1,q+1}$.
\end{Lem}
\begin{proof}
Lemma~\ref{pointwiseStabilizers} tells us that the edges in our cell complex have cyclic stabilisers, and that only those of type $\Z/2$ can contribute nontrivially to the $\Z/2$--modules $\tilde{E}^1_{1,q}$, $\tilde{E}^1_{1,q+1}$ and $\left(E^1_{1,q}\right)_{(2)} $.
Then we see from lemma~\ref{finiteSubgroupsHomology} and the Universal Coefficient Theorem that 
\mbox{$\tilde{E}^1_{1,q} \cong \tilde{E}^1_{1,q+1} \cong \left(E^1_{1,q}\right)_{(2)} $.}
Consider matrices for the homomorphisms $d^1_{1,q} \otimes \Z/2$, $\tilde{d}^1_{1,q}$ and $\tilde{d}^1_{1,q+1}$. Then applying lemma~\ref{inducedMaps} with both the coefficients $M = \Z$ and $M = \Z/2$, we check entry by entry that these matrices are identical.
\end{proof}

With the following lemma, we further find out that in degrees $q > 2$, the only possible solution to the d\'evissage problem is the trivial solution. An instance of non-trivial d\'evissage at degree $q = 2$ is given in subsection~\ref{results}. 
\begin{lemma} \label{direct sum decomposition}
Let $q > 2$, $\ell = 2$ or $\ell = 3$, and  $\ringO_{-m}$ any imaginary quadratic ring. Then the $\ell$--primary part of the homology $\Homol_q(\PSLO;\Z)$ is the direct sum over the $(E^\infty_{p,q-p})_{(\ell)}$--terms, the index~$p$ running from $0$ to $2$. 
\end{lemma}
\begin{proof} 
First observe that the $E^2$-page is concentrated in the first three columns $p \in \{0, 1, 2\}$, because of 
the existence of a 2-dimensional equivariant retract.
\begin{itemize}
\item Let $\ell = 3$. The rows of the $E^1$-page with $q$ even and $q \geq 2$ do not contain any 3-torsion, so neither does the $E^\infty$-page. So the assertion follows knowing from lemma~\ref{pointwiseStabilizers} that the $E^1$-page is concentrated in the first two columns for $q > 0$.
\item Let $\ell = 2$. Lemma \ref{mod2ranks} implies that $\tilde{E}^\infty_{1,q} \cong \tilde{E}^\infty_{1,q+1} \cong \left(E^\infty_{1,q}\right)_{(2)} $ and \\
$\tilde{E}^\infty_{0,q} \cong \tilde{E}^\infty_{0,q+1} \cong \left(E^\infty_{0,q}\right)_{(2)} $. Hence the only possible solution to the d\'evissage problem is the trivial solution, as we have claimed.
\end{itemize}
\end{proof}

Using the well-known fact that the virtual cohomological dimension (abbreviated vcd, see~\cite{Brown} for the definition of this) of the Bianchi groups is 2, we immediately obtain the following statement from this lemma.
\begin{thm}
Let $q > 2$. Then for $\ringO_{-m}$ any imaginary quadratic ring, \\
$\Homol_q(\PSLO;\Z)$ is a direct sum of copies of $\Z/2$ and $\Z/3$.
\end{thm}
It is not an obvious guess that no greater torsion appears here, because the Bianchi groups admit finite index subgroups with arbitrarily large primes occurring as orders of elements of their first homology group \cite{Sengun}.

\section{Torsion subcomplex reduction} \label{axisGraphReduction}

Let $Z$ be a $\Gamma$-cell complex, and let $\ell$ be a prime number.
\begin{df}
The \emph{$\ell$-torsion subcomplex} is the subcomplex of $_\Gamma \backslash Z$ consisting of all the cells, the pre-images of which have stabilisers in $\Gamma$ containing elements of order $\ell$. 
\end{df}

We immediately see that for $Z$ the refined cellular complex, $\Gamma = \PSLO$ and $\ell$ one of the two occuring primes $2$ and $3$, this subcomplex is a finite graph, because by lemma~\ref{pointwiseStabilizers}, the cells of dimension greater \mbox{than 1} are trivially stabilised in the refined cellular complex. 

\begin{lemma} \label{coincidence}
Except for the two cases $m = 1$, $\ell = 2$ and $m = 3$, $\ell = 3$, the $\ell$--torsion subcomplexes of the Fl\"oge cellular complex and the refined cellular complex coincide for all the Bianchi groups.
\end{lemma}
\begin{proof}
The only cells which the Fl\"oge cellular complex can admit outside the refined cellular complex, are the cusps associated to non-trivial ideal classes, of stabiliser type~$\Z^2$. So the first of these $\ell$--torsion subcomplexes is always contained in the second. 
On the other hand, the cells of the Bianchi fundamental polyhedron which are contracted to obtain the Fl\"oge cellular complex, are the 3--cell filling out the interior of the Bianchi fundamental polyhedron and its vertical facets, which all touch the cusp~$\infty$. As these cells are stabilised pointwise, their stabilisers can only contain $\ell$--torsion if the stabiliser of the cusp~$\infty$ contains $\ell$--torsion. The stabiliser of the cusp~$\infty$ consists of matrices of the form \small $\begin{pmatrix} \mu & \lambda \\ 0 & \frac{1}{\mu} \end{pmatrix}$ \normalsize with entries in $\ringO_{-m}$, so the entry $\mu$ must be a unit in $\ringO_{-m}$.
The only imaginary quadratic rings admitting units other than $\pm 1$, are the Gaussian and Eisenstein integers, so for such a matrix to be an $\ell$--torsion element, only the cases $m = 1$, $\ell = 2$ and $m = 3$, $\ell = 3$ occur.
\end{proof}

The cells which the 2-torsion and the 3-torsion subgraphs have in common, are precisely the vertices of stabiliser types $\Sthree$ and $\Afour$.
For example, we find the 2-torsion subgraph drawn in 
dashed lines~(\begin{pspicture}(-0.251,0.4)(0.3,0.6)
        \psset{linecolor=green,linestyle = dashed, linewidth=1.4pt}
        \psline(-0.2,0.5)(0.2,0.5)
\end{pspicture}) and the 3-torsion subgraph drawn in 
dotted lines~(\begin{pspicture}(-0.251,0.4)(0.3,0.6)
        \psset{linecolor=blue,linestyle = dotted, linewidth=2pt} 
        \psline[dotsep = 0.5pt](-0.2,0.5)(0.2,0.5)
\end{pspicture}) 
in the fundamental domain for the Fl\"oge cellular complex in figure \ref{FloegeFundamentalDomains}, where the vertices with matching labels are to be identified.


Knowing the types of the cell stabilisers which can appear, (the finite groups of lemma~\ref{finiteSubgroups} and the cusp stabilisers of type $\Z^2$), we immediately see that the $\ell$--primary part of the $E^1_{p,q}$--terms of the equivariant spectral sequence converging to the integral homology of the Bianchi groups,
$$
\left(E^1_{p,q}\right)_{(\ell)}=\bigoplus_{\sigma\thinspace\in\thinspace _\Gamma\backslash X^p}\left(\Homol_q(\Gamma_\sigma; \thinspace \Z)\right)_{(\ell)},
$$
depends for $ q \geq 1$ only on the stabilisers of cells in the $\ell$-torsion subcomplex.
Similarly for the  $\ell$--primary part of the associated differential $(d^1_{1,q})_{(\ell)}$ : for $ q \geq 1$, all cell inclusions which contribute non-trivially to it, can be found on the $\ell$-torsion subcomplex.


Now we use the geometric rigidity statements of section~\ref{Rigidity of the action on the refined cell complex} to fuse cells in the torsion subcomplexes.
Let $(a,v)$ and $(v',b)$ be adjacent edges in the $\ell$-torsion subcomplex. This means, they are adjacent to a common vertex orbit $\Gamma \cdot v = \Gamma \cdot v'$.
 
\begin{df} \label{edgeFusion}
If there are exactly two edges adjacent to the vertex~$v$ in the $\ell$-torsion subcomplex, then we define an \emph{edge fusion} by replacing the edges $(a,v)$ and $(v',b)$ by the edge $(a,b)$ and forgetting the vertex~$v$.
\end{df}

\begin{df}
We will call a \emph{reduced $\ell$-torsion subcomplex}, a cell complex obtained from the $\ell$-torsion subcomplex of the refined cellular complex by iterating edge fusions as often as this is permitted by definition~\emph{\ref{edgeFusion}}.
\end{df}


\begin{lemma} \label{invarianceUnderReduction}
In all rows $q \geq 1$, the $E^2$-page of the equivariant spectral sequence converging to \\
 $\Homol_{p+q}(\PSLO; \thinspace \Z)$ is invariant under replacing the $\ell$-torsion subcomplex by a reduced $\ell$-torsion subcomplex for the computation of the $\ell$-primary part of the differential $d^1_{1,q}$.
\end{lemma}
\begin{proof}
A vertex representative $v$ which is removed by an edge fusion must have exactly two orbits of edges of stabiliser type $\Z/ \ell$ adjacent to it. Lemma~\ref{geometricRigiditytheorem} tells us that then, $\Gamma_v$ is isomorphic to $\Z/2$ or $\Sthree$ in the case $\ell = 2$, and to $\Z/3$ or $\Afour$ in the case $\ell = 3$.
Now we see from definition~\ref{edgeFusion} and lemma~\ref{finiteSubgroupsHomology} that every edge fusion decreases each by $1$ the $\Z/ \ell$--ranks of the modules $E^1_{1,q}$ and $E^1_{0,q}$ for odd $q$, because there is a unique isomorphism type of the $\ell$-primary part of the homology of the stabilisers for vertices with two adjacent edges in the  $\ell$-torsion subcomplex:
\begin{itemize}
\item For the 2-primary part, $\Homol_q(\Sthree; \Z)_{(2)} \cong \Z/2 \cong \Homol_q(\Z/2; \Z)_{(2)}$, for odd $q$,
\item for the 3-primary part, $\Homol_q(\Afour; \Z)_{(3)} \cong \Z/3 \cong \Homol_q(\Z/3; \Z)_{(3)}$, for odd $q$,
\end{itemize}
and the above $\ell$-primary parts are all zero for $q>0$ even. 

Let $q$ be odd. We will show that any edge fusion also decreases by $1$ the $\Z/ \ell$--rank of  $(d^1_{1,q})_{(\ell)}$.
 Then we can conclude that the $E^2$-page is preserved under each edge fusion.
\\                                                                            
Let $\Graph '$ be the graph obtained by an edge fusion from an $\ell$-torsion subgraph $\Graph$.
We will show that passing from $\Graph '$ to $\Graph$ increases the rank of $(d^1_{1,q})_{(\ell)}$ by $1$.
Denote by $(a,b)$ the fusioned edge. 
There is a column associated to it in $(d^1_{1,q})_{(\ell)}$ for $\Graph '$ of the shape \scriptsize
$$\begin{array}{l|c}
& (a,b) \\
\hline & \\
a & -1  \\
b & 1 
\end{array} $$ \normalsize
in a suitable choice of bases for the $\ell$-primary part of the homology groups of the stabilisers. 
The remaining entries in this column are zeroes.
\\
In the graph $\Graph$, we have an additional vertex~$v$; 
and we replace our edge by the two edges $(a,v)$ and $(v',b)$, where $v'$ is on the same orbit as $v$. 
By what we have seen in the beginning of this proof, $\Homol_q(\Gamma_v; \Z)_{(\ell)} \cong \Z/ \ell$.
Furthermore, lemma~\ref{inducedMaps} tells us that the inclusions of the stabilisers of $(a,v)$ and $(v',b)$ into $\Gamma_v$ induce injections on homology.
Hence passing to $\Graph$, we replace the above matrix column by two columns of the shape \scriptsize
$$\begin{array}{l|cc}
& (a,v) & (v',b) \\
\hline && \\
v & 1 & -1 \\
a & -1 & 0 \\
b & 0 & 1 
\end{array} $$ \normalsize
(again in a suitable choice of bases, and the rest of these columns are zeroes). 
As the vertex~$v$ has exactly two edges adjacent to it in the $\ell$-torsion subgraph, 
the remaining entries in the inserted row associated to the vertex~$v$ are zeroes.
We further observe that the sum of the two inserted columns equals the replaced column (after concatenating a zero entry in the row of $v$). So the differential $(d^1_{1,q})_{(\ell)}$ for $\Graph$ has the same rank as the matrix 
$$\begin{array}{l|cc}
& (a,v) &  (a,v)+(v',b) {\rm \medspace and \medspace remaining \medspace columns}  \\
\hline && \\
\medspace \thinspace v & 1 & 0  \\
\begin{array}{l} a \\ {\rm other \medspace rows} \end{array} & 
\begin{array}{c} -1 \\ 0 \end{array} & 
(d^1_{1,q})_{(\ell)} \medspace {\rm for } \medspace \Graph ' \\
\end{array} $$
Hence the rank of the $\ell$-primary part of the differential $d^1_{1,q}$ has increased exactly by~1.
\end{proof}

The geometrical meaning of a reduced torsion subgraph is the following.
\begin{rem}
From the proof of lemma~\ref{geometricRigiditytheorem}, we see that for $\PSLO$ and the refined cell complex, any pair of fusioned edges has pre-images that lie on the same rotation axis. On the other hand, the quotient of any axis for rotations of order $\ell$  is a  chain of fusionable adjacent edges in the $\ell$--torsion subgraph.
Hence a reduced $\ell$-torsion subgraph contains one edge for every $\PSLO$--representative of axes for rotations of order $\ell$.
\end{rem}

\subsection{Classifying the reduced torsion subcomplexes} ${}$
\\
Given an $\ell$-torsion subgraph for $\PSLO$, the only difference that can occur between two of its reductions,  is the following.
If there is a loop in the graph, then this loop will become a single edge with identical origin and end vertex. But this vertex can be chosen arbitrarily from the vertices which are originally on the loop.
However, the topology of the reduced graph does not depend on this choice of vertex.
So as a topological space, it is well defined to speak of \emph{the} reduced $\ell$-torsion subgraph.
The following key lemma, together with lemma~\ref{direct sum decomposition} and the fact from lemma~\ref{pointwiseStabilizers} that the 2-cells are trivially stabilised, directly implies theorem~\ref{HomologicalInvarianceUnderHomeomorphisms}.
\begin{lemma} \label{E2invarianceUnderHomeomorphisms}
The $\ell$--primary part of the terms $E^2_{p,q}$ of the equivariant spectral sequence converging to  $\Homol_{p+q}(\PSLO; \thinspace \Z)$ depends in all rows $q \geq 1$ only on the homeomorphism type of the $\ell$--torsion subcomplex.
\end{lemma}
To prove this, we need several sub-lemmata. First, we discard the exceptional cases coming from the additional units in the rings of Gaussian and Eisenstein integers.
\begin{sublem} \label{exceptions}
 The homeomorphism types of the $\ell$--torsion subgraphs for $m = 1$, $\ell = 2$ and $m = 3$, $\ell = 3$ are unique among the $\ell$--torsion subgraphs of the Bianchi groups. 
\end{sublem}
\begin{proof}
 By lemma \ref{coincidence}, for each of the occurring primes $\ell = 2$ and $\ell = 3$, there is only one case where the $\ell$--torsion subgraph of the Fl\"oge cellular complex does not coincide with the $\ell$--torsion subgraph of the refined cellular complex.
In this case, the second subgraph is not closed, as it has edges reaching out to the cusp~$\infty$.
So it cannot be homeomorphic to the $\ell$--torsion subgraph~$\Graph$ for any other~$m$, because cells are fixed pointwise and hence all vertices of edges of $\Graph$ not reaching out to cusps are contained in $\Graph$.
\end{proof}

\begin{sublem} \label{directSumDecomposition}
The matrix for the $\ell$-primary part of $d^1_{1,q}$ can be decomposed as a direct sum of the blocks associated to the connected components of the $\ell$-torsion subgraph.
\end{sublem}
\begin{proof}
As there is no adjacency between different connected components, all entries off these blocks are zero.
\end{proof}

\begin{sublem} \label{key lemma}
 The reduced $\ell$--torsion subgraph $\Graph$ for $\PSLO$, together with the associated stabiliser types --- up to a choice which does not influence the $E^2$-page --- is determined by the homeomorphism type of $\Graph$.
\end{sublem}
\begin{proof}
By sub-lemma~\ref{directSumDecomposition}, we only need to check this on each connected component of $\Graph$.
\begin{itemize}
\item Consider a connected component of $\Graph$ which is an isolated loop.
 Such a loop consists of one edge with its two endpoints identified in $\Graph$.
 The stabilisers of the pre-images of the endpoints are of type $\Z/2$ or $\Sthree$ in the case $\ell = 2$,
 and of type $\Z/3$ or $\Afour$ in the case $\ell = 3$.
 As stated in the proof of lemma~\ref{invarianceUnderReduction},
 the $\ell$--primary part of the homology groups is the same for both possible stabiliser types,
 and by lemma~\ref{inducedMaps} the inclusions of $\Z/\ell$ into them induce always injections,
 so we do not need to know which is precisely this stabiliser type
 to reobtain the original contribution to the $E^2$-page.
\item Consider a connected component of $\Graph$ which is not an isolated loop. On such a connected component,
 the reduction has left only vertices with one or three edges adjacent to them:
 we have eliminated all vertices with two adjacent edges by edge fusions;
 and vertices with no adjacent edges cannot be in the $\ell$-torsion subgraph
 due to lemma~\ref{geometricRigiditytheorem}.
By the latter lemma, there is a unique stabiliser type of vertices with one adjacent edge in the $\ell$--torsion subgraph. Furthermore, the group $\Kleinfourgroup$ is the unique stabiliser type of vertices with three adjacent edges in the 2-torsion subgraph; and in the 3-torsion subgraph, there is no vertex with three adjacent edges at all.
We can recognise vertices with one, respectively three adjacent edges as end points respectively bifurcation points in our connected component of $\Graph$, considered as a topological space $Y$. The end points and bifurcation points are preserved by homeomorphisms. So, our connected component of $\Graph$ can be reconstructed from the homeomorphism type of $Y$, as well as the associated stabiliser types.
\end{itemize}
\end{proof}

\begin{sublem} \label{E2 terms}
 The $\ell$--primary part of the terms $E^2_{p,q}$, in all cases $q \geq 1$, is determined by the reduced $\ell$--torsion subgraph and the associated stabiliser types.
\end{sublem}
\begin{proof}
 The blocks in the $\ell$--primary part of the terms $d^2_{p,q}$ differential matrices are, in all cases $q \geq 1$, determined completely by the involved stabiliser types, as we see from the discussion in sections \ref{3primary} and~\ref{2primary}, lemma~\ref{D2blocks}  and lemma~\ref{inducedMaps}.
\end{proof}

\begin{proof}[Proof of lemma {\rm \ref{E2invarianceUnderHomeomorphisms}}.]
By sub-lemma \ref{exceptions}, we need only consider the cases where we can identify the $\ell$--torsion subgraphs of the Fl\"oge cellular complex and the refined cellular complex using lemma~\ref{coincidence}.
By lemma~\ref{invarianceUnderReduction}, we can pass from the $\ell$-torsion subgraph to a reduced $\ell$--torsion subgraph.
Now we apply sub-lemma~\ref{key lemma} and sub-lemma~\ref{E2 terms}.
\end{proof}

\section{Group homology computations}
\begin{figure} 
  \caption{ Fundamental domains for $\PSLO$ in the Fl\"oge cellular complex, in the non-Euclidean principal ideal domain cases. The vertices with matching labels are identified by $\PSLO$, and then the 2-torsion subcomplex is obtained from the dashed edges and the 3-torsion subcomplex from the dotted edges.}
\label{FloegeFundamentalDomains}
  \centering 
  \subfloat[][\mbox{$m=19$}]{\includegraphics[width=26mm]{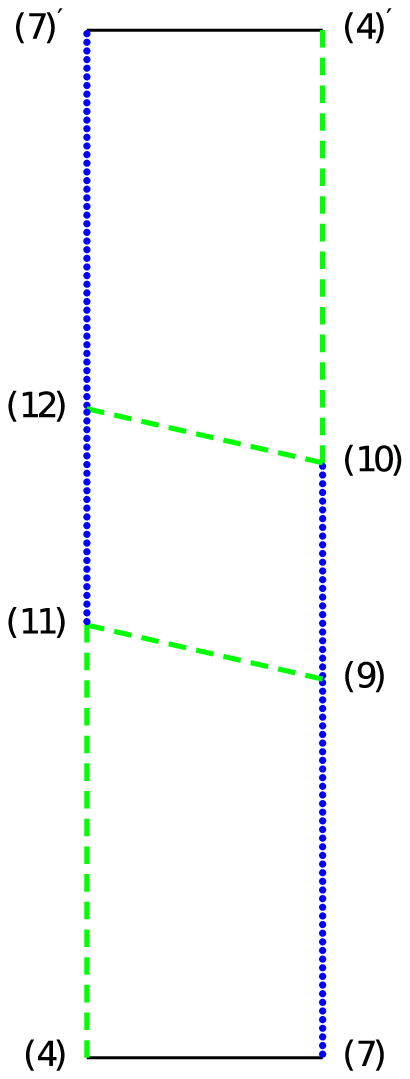}}
\qquad
  \subfloat[][\mbox{$m=43$}]{\includegraphics[width=27mm]{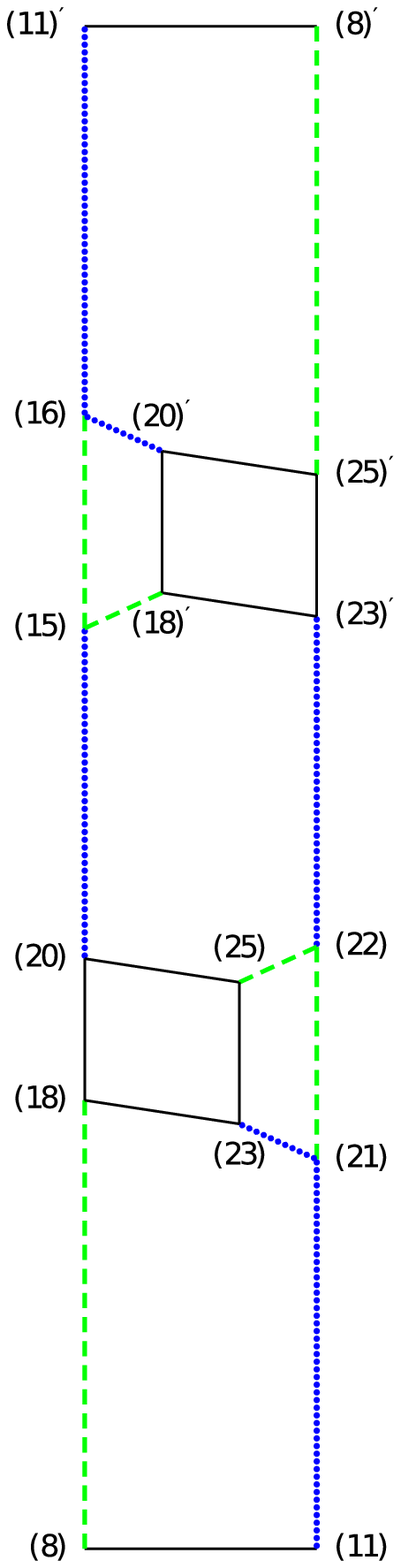}}
\qquad
  \subfloat[][\mbox{$m=67$}]{\includegraphics[width=32mm]{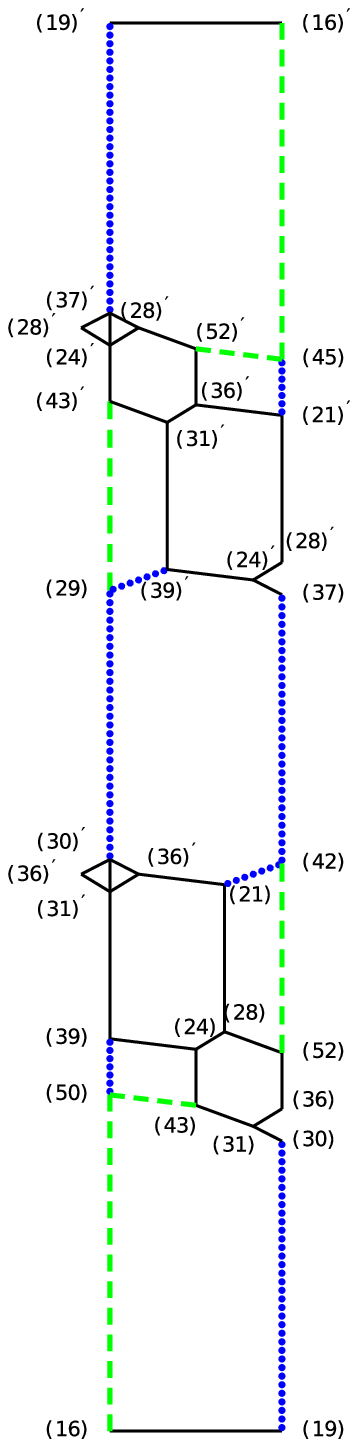}}  
\qquad              
  \subfloat[][\mbox{$m=163$}]{\includegraphics[width=35mm]{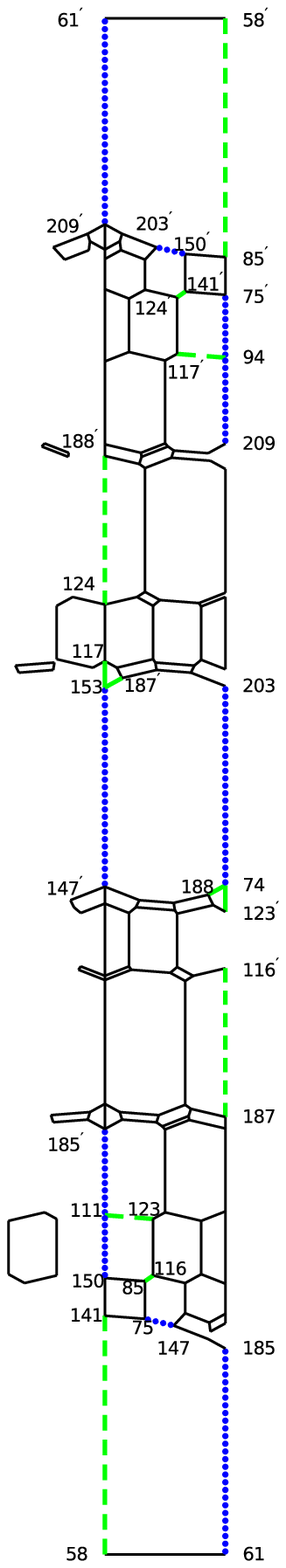}}
\end{figure}

We now give the intermediate results in the group homology computation, in the case $m=67$. Our fundamental domain for the Fl\"oge cellular complex (which coincides with Mendoza's spine in the principal ideal domain cases) is drawn in figure~\ref{FloegeFundamentalDomains}(C). We denote by $(k)$ the vertex number $k$ in the output files of the program \cite{BianchiGP}, and by $\Gamma_{(k)}$ its stabiliser. We do the same for the edges, which we denote by their endpoints.
We will write $(k)'$ for the other vertices on the same $\Gamma$-orbit as $(k)$.
We will use the following notations:
\begin{alignat*}4
& A := \pm \scriptsize
\left(  \begin{array}{*{2}{c}}
- \omega
 & 
16
 \\ 
1
 & 
1
+ \omega
\end{array} \right) , &\qquad &
 \normalsize B:= \pm \scriptsize
\left(  \begin{array}{*{2}{c}}
- \omega
 & 
8
 \\ 
2
 & 
1
+ \omega
\end{array} \right) , 
\\
& C := \pm \scriptsize
\left(  \begin{array}{*{2}{c}}
5
- \omega
 & 
10
+2 \omega
 \\ 
2
+ \omega
 & 
-6
+ \omega
\end{array} \right) \normalsize, 
&\qquad &
D := \pm \scriptsize
\left(  \begin{array}{*{2}{c}}
13
 & 
10
+10 \omega
 \\ 
 \omega
 & 
-13
\end{array} \right) , 
\\
& \normalsize F:= \pm \scriptsize 
\left(  \begin{array}{*{2}{c}}
- \omega
 & 
4
 \\ 
4
 & 
1
+ \omega
\end{array} \right) ,
&\qquad &
 \normalsize G:= \pm \scriptsize
\left(  \begin{array}{*{2}{c}}
-2
- \omega
 & 
4
- \omega
 \\ 
3
 & 
2
+ \omega
\end{array} \right) , \\
& \normalsize H:= \pm \scriptsize
\left(  \begin{array}{*{2}{c}}
-1
- \omega
 & 
15
- \omega
 \\ 
1
 & 
1
+ \omega
\end{array} \right) ,
&\qquad &
\normalsize J := \pm \scriptsize
\left(  \begin{array}{*{2}{c}}
 & 
1
 \\ 
-1
 & 
\end{array} \right) , 
\\ &
 \normalsize L:= \pm \scriptsize
\left(  \begin{array}{*{2}{c}}
-1
- \omega
 & 
7
- \omega
 \\ 
2
 & 
2
+ \omega
\end{array} \right) , &\qquad &
\normalsize N:= \pm \scriptsize
\left(  \begin{array}{*{2}{c}}
9
- \omega
 & 
14
+6 \omega
 \\ 
2
+ \omega
 & 
-10
+ \omega
\end{array} \right),
\\ &
  \normalsize S := \pm \scriptsize
\left(  \begin{array}{*{2}{c}}
 & 
-1
 \\ 
1
 & 
1
\end{array} \right).
\end{alignat*} \normalsize
 We observe the following stabilisers of the vertex representatives.
$$ \small
\begin{array}{llll}
 \Gamma_{(36)} = \Gamma_{(24)} = & \Gamma_{(31)} = \Gamma_{(28)}  & = \{ 1 \}
\\
 \Gamma_{(43)} = & \langle JF  \thinspace | \thinspace (JF)^2 = 1 \rangle & \cong {\mathbb{Z} /2} 
\\
 \Gamma_{(16)} = &
\langle J  \thinspace | \thinspace J^2 = 1 \rangle 
& \cong {\mathbb{Z} /2} \\
 \Gamma_{(52)} = 
& \langle G  \thinspace | \thinspace G^2 = 1 \rangle
& \cong {\mathbb{Z} /2} \\
 \Gamma_{(42)} = 
&\langle L,\thinspace G  \thinspace | \thinspace L^3 = G^2 = (LG)^3 = 1 \rangle
& \cong {\Afour} \\
 \Gamma_{(29)} = 
&\langle B, \thinspace N \thinspace | \thinspace B^3 = N^3 = (BN)^2 = 1 \rangle
& \cong {\Afour} \\
 \Gamma_{(45)} = 
&\langle D , \thinspace H \thinspace | \thinspace D^2 = H^2 = (DH)^3 = 1 \rangle
& \cong {\Sthree} \\
 \Gamma_{(50)} = 
&\langle J, \thinspace F  \thinspace | \thinspace J^2 = F^3 = (JF)^2 = 1 \rangle 
& \cong {\Sthree} \\
 \Gamma_{(37)} = 
&\langle L  \thinspace | \thinspace L^3 = 1 \rangle
& \cong {\mathbb{Z} /3} \\
 \Gamma_{(21)} = 
&\langle C  \thinspace | \thinspace C^3 = 1 \rangle
 &\cong {\mathbb{Z} /3} \\
 \Gamma_{(39)} = 
&\langle F  \thinspace | \thinspace F^3 = 1 \rangle
& \cong {\mathbb{Z} /3} \\
 \Gamma_{(19)} = \Gamma_{(30)} =
&\langle S  \thinspace | \thinspace S^3 = 1 \rangle 
& \cong {\mathbb{Z} /3} .
\end{array}$$
\normalsize
The following edge representatives have stabiliser type $\Z/2$:
$$ \small
\begin{array}{llll}
\Gamma_{(42),(52)} = 
\langle G  \thinspace | \thinspace G^2 = 1 \rangle
\\
\Gamma_{(18),(45)} = 
\langle H  \thinspace | \thinspace H^2 = 1 \rangle
 \\
\Gamma_{(45),(51)} = 
\langle D  \thinspace | \thinspace D^2 = 1 \rangle
 \\
\Gamma_{(43),(50)} = 
\langle JF  \thinspace | \thinspace (JF)^2 = 1 \rangle
 \\
\Gamma_{(16),(50)} = 
\langle J  \thinspace | \thinspace J^2 = 1 \rangle 
 \\
\Gamma_{(29),(44)} = 
\langle BN  \thinspace | \thinspace (BN)^2 = 1 \rangle .
\end{array}
$$ \normalsize
In order to compare these with the vertex representative stabilisers,  we note that there are respectively two elements of $\Gamma$ sending the vertex $(18)$ to $(16)$, $(44)$ to $(43)$ and $(51)$ to $(52)$. 

The following edge representatives have stabiliser type $\Z/3$:
$$ \small
\begin{array}{llll}
\Gamma_{(19),(30)} =  
\langle S  \thinspace | \thinspace S^3 = 1 \rangle 
\\
\Gamma_{(21),(42)} = 
\langle C  \thinspace | \thinspace C^3 = 1 \rangle
\\
\Gamma_{(37),(42)} = 
\langle L  \thinspace | \thinspace L^3 = 1 \rangle
\\
\Gamma_{(26),(45)} = 
\langle DH \thinspace | \thinspace (DH)^3 = 1 \rangle
\\
\Gamma_{(20),(23)} = 
\langle A  \thinspace | \thinspace A^3 = 1 \rangle
\\
\Gamma_{(29),(54)} = 
\langle N  \thinspace | \thinspace N^3 = 1 \rangle
\\
\Gamma_{(29),(47)} = 
\langle B  \thinspace | \thinspace B^3 = 1 \rangle
\\
\Gamma_{(39),(50)} = 
\langle F  \thinspace | \thinspace F^3 = 1 \rangle
\end{array}
$$ 
In order to determine the homomorphisms into the vertex stabilisers, we give a matrix for each of the following vertex identifications. The whole coset of matrices performing this vertex identification is obtained by multiplying the edge stabiliser from the right onto this matrix.
$$
\begin{array}{{l|cccccc}}
\txt{The matrix} & \txt{sends vertex number} & \txt{to vertex number} \\ \\
\hline \\ \pm
\scriptsize
\left(  \begin{array}{*{2}{c}}
7
 & 
5
+5 \omega
 \\ 
 \omega
 & 
-12
\end{array} \right) \normalsize
& (26) 
& (21)
\\ \\ \pm
\scriptsize
\left(  \begin{array}{*{2}{c}}
1
 & 
 \omega
 \\ 
 & 
1
\end{array} \right) \normalsize
& (20) 
& (19) \\ \\ \pm \scriptsize
\left(  \begin{array}{*{2}{c}}
8
-2 \omega
 & 
30
+7 \omega
 \\ 
3
+ \omega
 & 
-12
+2 \omega
\end{array} \right)\normalsize
&  (23) 
&(37)
\\ \\ \pm
\scriptsize
\left(  \begin{array}{*{2}{c}}
5
 & 
3
+3 \omega
 \\ 
 \omega
 & 
-10
\end{array} \right) \normalsize
& (54) 
&(39)
\\ \\ \pm \scriptsize
\left(  \begin{array}{*{2}{c}}
-3
- \omega
 & 
6
- \omega
 \\ 
5
 & 
1
+2 \omega
\end{array} \right) \normalsize
& (47) 
& (30).
\end{array}
$$
The remaining seventeen edge orbits have trivial stabiliser. There are fifteen orbits of 2-cells.
The above cardinalities sum up to the equivariant Euler characteristic
$$\chi_\Gamma(X) =
4 +\frac{3}{2} +\frac{2}{12} +\frac{2}{6} +\frac{5}{3}
-17 -\frac{6}{2} -\frac{8}{3}
+15 = 0,$$
whence there is a check of our calculations in view of proposition \ref{Euler_characteristic_vanishes}.

\subsubsection*{The $d^1$--differentials in the equivariant spectral sequence}$\text{ }$ 
\\
On the 3-primary part, for $q \equiv 3 \mod 4$, \quad $(d^1_{1,q})_{(3)}$: $(\Z/3)^9 \gets (\Z/3)^8$ can be expressed by the matrix \begin{center} \scriptsize 
$\begin{pmatrix} -1& 0& 0& 0& -1& 0& 0& 0\\ 1& 0& 0& 0& 0& 0& 1& 0\\ 0& 2& 1& 0& 0& 0& 0& 0\\ 0& -1& 0& -1& 0& 0& 0& 0\\ 0& 0& -2& 0& 1& 0& 0& 0\\ 0& 0& 0& 1& 0& 0& 0& 0\\ 0& 0& 0& 0& 0& -1& -2& 0\\ 0& 0& 0& 0& 0& 1& 0& -1\\ 0& 0& 0& 0& 0& 0& 0& 1 \end{pmatrix}$ \normalsize of rank 8, 
\end{center}
and for $q \equiv 1 \mod 4$, we obtain a matrix for 
$(d^1_{1,q})_{(3)}$: $(\Z/3)^7 \gets (\Z/3)^8$ by omitting the sixth and the last row of the above matrix.
The resulting rank is 7.
So in both cases the highest possible rank occurs.

\bigskip

For $q > 1$ odd, the 2-primary part $(d^1_{1,q})_{(2)}$ has preimage $(\Z/2)^6$ and full rank 6, and can be expressed in terms of the matrix
\begin{center} $A :=$ \scriptsize $\begin{pmatrix} -1& 0& 0& 0& 0& 0\\ 1& 0& 1& 0& 0& 0\\ 0& 1& -1& 0& 0& 0\\ 0& 0& 0& -1& 0& 1\\ 0& -1& 0& 0& -1& 0\\ 0& 0& 0& 0& 0& -1\\  0& 0& 0& 1& 1& 0 \end{pmatrix}$ \normalsize \end{center}
in the following way. For the differential $(d^1_{1,6k+5})_{(2)}$, with target space $(\Z/2)^{7+2k+2}$, this is the matrix $A$ with $2k+2$ zero rows to be inserted.
For the differentials $(d^1_{1,6k+3})_{(2)}$  and $(d^1_{1,6k+7})_{(2)}$, with target space  $(\Z/2)^{7+2k}$, we obtain their matrix from $A$ by inserting $2k$ zero rows.

Finally, a matrix for $(d^1_{1,1})_{(2)}$: $(\Z/2)^5 \gets (\Z/2)^6$ is obtained by omitting the first and sixth rows from~$A$. This gives us a matrix of rank 5.
We deduce the $E^2$-page of proposition~\ref{non-EuclideanPrincipalE2page}.

\subsection{The non-Euclidean principal ideal domain cases}  \label{results}

We now give the results in the cases $m = 19,$ $43,$ $67$ and $163$, which are the non-Euclidean principal ideal domain cases. The Euclidean principal ideal domain cases are already known from \cite{SchwermerVogtmann}.
We observe that in these four cases, the torsion in the integral homology of $ {\rm PSL}_2(\ringO_{-m})$ is of the same isomorphism type.
This comes from the fact that their 2-torsion and 3-torsion subgraphs are homeomorphic (see figure~\ref{FloegeFundamentalDomains}).
Lemma~\ref{E2invarianceUnderHomeomorphisms} then explains this isomorphism.

\begin{prop} \label{non-EuclideanPrincipalE2page}
For $m \in \{ 19, 43, 67, 163\}$, the $E^2$-page of the equivariant spectral sequence is concentrated in the columns $p=0$, $p=1$ and $p=2$, given as follows. \footnotesize
$$
\begin{array}{{l|ccccc}}
q = 12n + 14 & (\Z/2)^{4n + 6}\\
q = 12n + 13 & (\Z/2)^{4n + 3} & & \Z/3\\
q = 12n + 12 & (\Z/2)^{4n + 4}\\
q = 12n + 11 & (\Z/2)^{4n + 5} \oplus \Z/3 & \\
q = 12n + 10 & (\Z/2)^{4n + 2}\\
q = 12n + 9 & (\Z/2)^{4n + 3} & & \Z/3\\
q = 12n + 8 & (\Z/2)^{4n + 4}\\
q = 12n + 7 & (\Z/2)^{4n + 1} \oplus \Z/3 & \\
q = 12n + 6 & (\Z/2)^{4n +2}\\
q = 12n + 5 & (\Z/2)^{4n + 3} & & \Z/3\\
q = 12n + 4 & (\Z/2)^{4n}\\
q = 12n + 3 & (\Z/2)^{4n + 1} \oplus \Z/3 & \\
q = 2       & (\Z/2)^2 \\
q = 1       &  0 & &  \Z/2 \oplus \Z/3 \\
q = 0       &  \Z & &    \Z^{\beta_1} & &   \Z^{\beta_2},
\\
\\ \hline 
& p = 0 & & p = 1 & & p = 2
\end{array}
$$ \normalsize
where  \scriptsize
$\begin{array}{l|ccccc}
m        &  19 & 43 & 67 & 163 \\
\hline 
\beta_1  &  1   & 2  &  3 & 7   \\
\end{array}$ \normalsize
gives the Betti number $\beta_1$, and $\beta_2 = \beta_1 -1$.
\end{prop}
The Betti numbers are related by the vanishing of the naive Euler characteristic \cite{Vogtmann}, 
\\
$\beta_0 -\beta_1 +\beta_2 = 0$.

We observe that the $E^2_{0,1}$-term vanishes completely. This term is the target of the only $d^2$-arrow which can for arbitrary $m$ be non-zero, namely $d^2_{2,0}$. Hence our spectral sequence degenerates at the $E^2$-level.
The only ambiguity in the d\'evissage concerns $\Homol_2({\rm PSL}_2(\ringO_{-m}); \Z)$; the above $E^\infty$-page says that its 2-primary part is either $(\Z/2)^3$ or $\Z/4 \oplus \Z/2$. Computing
$$ \dim_{\F_2} \Homol_q({\rm PSL}_2(\ringO_{-m}); \Z/2) = \scriptsize \begin{cases}
4k + 5, & q = 6k+8, \\ 
4k + 3, & q = 6k+7, \\ 
4k + 5, & q = 6k+6, \\ 
4k + 3, & q = 6k+5, \\ 
4k + 1, & q = 6k+4, \\ 
4k + 3, & q = 6k+3, \\ 
\beta_2 +2, & q = 2, \\
\beta_1, & q = 1,
\end{cases} $$ \normalsize
and comparing with the help of the Universal Coefficient Theorem, we can exclude the first possibility.
We obtain proposition \ref{non-Euclidean results}.

\begin{notation} \label{Poincare series}
Let $\ell$ be a prime number. Consider the Poincar\'e series in the dimensions over the field with $\ell$ elements, of the homology with $\Z/\ell$--coefficients of $\PSLO$,
$$P^\ell_m(t) := \sum\limits_{q \thinspace = \thinspace 3}^{\infty} \dim_{\mathbb{F}_\ell} \Homol_q \left(\text{PSL}_2\bigl(\mathcal{O}_{-m}\bigr);\thinspace \Z/\ell \right)\thinspace t^q.$$ 
\end{notation}

\subsection{Gaussian and Eisenstein integers} \label{Gaussian and Eisenstein integers}
In \cite{RahmThesis}, the computation for the cases of the Gaussian integers $\ringO_{-1} = \Z[\sqrt{-1}]$ and the Eisentein integers $\ringO_{-3} = \Z[\omega]$, with 
$\omega = -\frac{1}{2} +\frac{1}{2}\sqrt{-3} = {\rm e}^\frac{2\pi i}{3}$, has been redone by hand, following step by step the description in~\cite{SchwermerVogtmann}. This enables us to clean up some typographical impacts of the publication process (presumably the recomposition) to the results in the editor's version of the latter paper. Some of the implied corrections have already been suggested by Berkove~\cite{BerkoveMod2}.

For the Gaussian integers, the integral homology of ${\rm PSL}_2(\ringO_{-1})$
is a direct sum of copies of $\Z/2$ and $\Z/3$, with the number of copies specified by 
$$ \Homol_q({\rm PSL}_2(\ringO_{-1}); \Z) \cong
\begin{cases}
(\Z/2)^2 \oplus \Z/3, & q = 2, \\
(\Z/2)^2, & q = 1 \\ 
\end{cases}$$ 
and the Poincar\'e series $ P^2_1(t) = \frac{-2t^3(2t^3 - t^2 - 3)}{(t-1)^2 (t^2 + t + 1 ) }$
and \mbox{$ P^3_1(t) = \frac{-t^3(t^2 - t + 2)}{(t-1)(t^2+1)}$.}

For the Eisenstein integers,
$$ \Homol_q({\rm PSL}_2(\ringO_{-3}); \Z) \cong
\begin{cases}
\Z/4 \oplus \Z/2, & q = 2, \\
\Z/3, & q = 1  \\ 
\end{cases} $$
and $ \Homol_q({\rm PSL}_2(\ringO_{-3}); \Z)$ is for $q \geq 3$ a direct sum of copies of $\Z/2$ and $\Z/3$ with the number of copies specified by the Poincar\'e series $ P^2_3(t) = \frac{-t^3(t^3 - 2t^2 + 2t - 3)}{(t-1)^2 (t^2 + t + 1 ) }$
and \mbox{$ P^3_3(t) = \frac{-t^3(t^2 +2)}{(t-1)(t^2 +1)}$.}

\subsection{Computations of the homological torsion} \label{Results for the homological torsion}

\begin{figure} \caption{The results for the Poincar\'e series $ P^2_m(t)$ in 2--torsion}
\label{2torsionsubgraphs} 
\begin{tabular}{c|c|c}
$m$ & \begin{tabular}{c} {reduced} \\ { 2-torsion subgraph} \end{tabular} & $P^2_m(t)$
\\ 
\hline &  & 
 \\
\begin{tabular}{c} 7, 15, 35, 39, 87, 91, 95, \\ 115, 151, 155, 159, 191, 403 \end{tabular} &  \circlegraph & $\frac{-2t^3}{t-1}$ 
\\ &  & \\
46 & \circlegraph \circlegraph &  $2 \left(\frac{-2t^3}{t-1}\right)$ \normalsize
\\ &  & \\
235, 427 &   \circlegraph \circlegraph \circlegraph &  $3 \left( \frac{-2t^3}{t-1} \right)$ \normalsize
\\ & &\\
3, 11, 19, 43, 67, 139, 163  &  \edgegraph &   $\frac{-t^3(t^3 - 2t^2 + 2t - 3)}{(t-1)^2 (t^2 + t + 1 ) }$ \normalsize
\\ & & \\
51, 123, 187, 267 &  \edgegraph \edgegraph &   $2 \left(\frac{-t^3(t^3 - 2t^2 + 2t - 3)}{(t-1)^2 (t^2 + t + 1 ) }\right)$ \normalsize
\\ & & \\
6, 22 & \circlegraph \edgegraph & $\frac{-2t^3}{t-1} +\frac{-t^3(t^3 - 2t^2 + 2t - 3)}{(t-1)^2 (t^2 + t + 1 ) }$ \normalsize
\\ &  & \\
5 , 10, 13, 29, 58 & \graphFive &  $\frac{-t^3(3t -5)}{(t-1)^2}$ \normalsize
\\ &  &\\
37 & \graphFive \circlegraph \circlegraph   & $\frac{-t^3(3t -5)}{(t-1)^2} +2 \left(\frac{-2t^3}{t-1}\right)$ \normalsize
\\ &  &\\
2 &  \graphTwo &   $\frac{-2t^3(t^3 - t^2 - 2)}{(t-1)^2 (t^2 + t + 1 ) }$ \normalsize
\\ &  &\\
34 &  \graphTwo \graphTwo \circlegraph \circlegraph  &  $ 2 \left( \frac{-2t^3(t^3 -t^2 -2)}{(t-1)^2 (t^2 + t + 1 ) }\right) +2\left(\frac{-2t^3}{t-1}\right)$ \normalsize
\\ &  &\\
\end{tabular}
\end{figure}
\normalsize

For the discussion in the rest of this section, we leave the two special cases $m = 1$, $\ell = 2$ and $m = 3$, $\ell = 3$ excluded, already having treated them in subsection~\ref{Gaussian and Eisenstein integers}.

\begin{observation} \label{homeomorphyTypesin3torsion}
Consider the case $\ell = 3$. By lemma~\ref{geometricRigiditytheorem}, there can be no bifurcation point in the 3-torsion subgraph. Hence, every connected component of a reduced 3--torsion subgraph consists of a single edge, 
\begin{itemize}
\item either with two vertices of stabiliser type $\Sthree$, 
\item or with identified end points of stabiliser type $\Z/3$ or $\Afour$, so this connected component is a loop.
\end{itemize}
In the second case, the contribution to the  $E^2$-page does not depend on the occurring stabiliser type,  as we see from lemma~\ref{E2invarianceUnderHomeomorphisms}.
\end{observation}

The Poincar\'e series of notation \ref{Poincare series} depends only on the $\ell$-primary part of the $E^2$-page because we have cut off the degrees smaller than or equal to the virtual cohomological dimension of~$\Gamma$, namely~2.

We observe that we can decompose the 3-torsion Poincar\'e series $P^3_m(t)$ as a sum over the series obtained from the connected components of the 3-torsion subgraph, because by sub-lemma~\ref{directSumDecomposition}, there can be no interference between different connected components.

\begin{observation} \label{linearDecomposition}
Hence it suffices to compute the 3-torsion Poincar\'e series $P^{3,1}(t)$ and $P^{3,2}(t)$ associated to the first and the second homeomorphism type appearing in 
\mbox{observation~\ref{homeomorphyTypesin3torsion},} and for any Bianchi group $\text{PSL}_2\bigl(\mathcal{O}_{-m}\bigr)$ count the numbers $n_1$ of connected components of first type, $n_2$ of connected components of second type.
Then the 3-torsion Poincar\'e series associated to this Bianchi group equals
$$P^3_m(t) = n_1 P^{3,1}(t) +n_2 P^{3,2}(t).$$
As $P^{3,1}(t)$ and $P^{3,2}(t)$ are linearly independent, the reduced $3$-torsion subgraph can be easily computed from the $3$-torsion Poincar\'e series.
\end{observation}

For the connected components in the 2-torsion subgraph, a priori infinitely many homeomorphism types may occur. 
But we still have, by \mbox{sub-lemma~\ref{directSumDecomposition}}, a direct sum decomposition of the 2-primary part of the $E^2$-page, for $q$ greater than the virtual cohomological dimension of the Bianchi group.

 The following lemma is useful in order to transform the Poincar\'e series of notation \ref{Poincare series} into fractions of finite polynomials in $t$. The results are given in figures~\ref{2torsionsubgraphs} and~\ref{3torsionsubgraphs}.

\begin{Lem} \label{PoincareSeries}
The equation
$\sum\limits_{k=0}^\infty (a k +b)t^{i k+j} = \frac{t^j (b+t^i (a -b))}{(1-t^i)^2}$ holds for all $a, b, i,j \in \N$.
\end{Lem}

\begin{proof}
$$\begin{array}{c}
(1-t^i)^2 \sum\limits_{k=0}^\infty (a k +b)t^{i k+j} \\
= \sum\limits_{k=0}^\infty (a k +b)t^{i k+j}
-2 \sum\limits_{k=0}^\infty (a k +b)t^{i (k+1)+j}
+\sum\limits_{k=0}^\infty (a k +b)t^{i (k+2)+j} \\
=  \sum\limits_{k=0}^\infty (a k +b)t^{i k+j}
-2 \sum\limits_{k=1}^\infty (a (k-1) +b)t^{i k+j }
+\sum\limits_{k=2}^\infty (a (k-2) +b)t^{i k +j} \\
\end{array}$$
Now the multiplicity in this term of $\sum\limits_{k=2}^\infty t^{i k +j} $
is $a k +b -2a k +2a -2b +a k -2a +b = 0$, so the above term equals
$b t^j +(a+b)t^{i+j} -2b t^{j+i} 
= t^j (b -b t^i +a t^i).$
\end{proof}

\section{Results for the special linear groups} \label{SL2}
In order to keep the $\ell$--torsion subcomplex low-dimensional, it is important to divide the arithmetic group by the subgroup generated by all the elements of order $\ell$  which are in the kernel of the action. 
Otherwise, when $\ell$ occurs as the order of an element in the kernel, 
the $\ell$--torsion subcomplex is the whole quotient complex.
For instance, for SL$_2(\ringO_{-m})$, where $\ringO_{-m}$ is a non-Euclidean principal ideal domain, we remark that for each 2--cycle in the quotient complex (which corresponds to a generator of $ \Homol_2({\rm SL}_2(\ringO_{-m}); \rationals) $ ), we have a constant summand $\Z/2$ to its integral homology in all degrees $q>1$.

\begin{proposition} 
The integral homology of ${\rm SL}_2(\ringO_{-m})$, \\ for \mbox{ $m \in \{ 19, 43, 67,163\}$,} is given as
$$ \Homol_q({\rm SL}_2(\ringO_{-m}); \Z) \cong 
\begin{cases}
(\Z/2)^{\beta_2} \oplus \Z/2 \oplus \Z/4 \oplus \Z/3, & q =  6 +4n, \\
(\Z/2)^{\beta_2}, & q = 5 +4n, \\
(\Z/2)^{\beta_2}\oplus \Z/2, & q = 4  +4n, \\
(\Z/2)^{\beta_2} \oplus \Z/2 \oplus \Z/8  \oplus \Z/3, & q = 3 +4n, \\
\Z^{\beta_2}   \oplus (\Z/2)^{\beta_2}\oplus \Z/2 \oplus \Z/4 \oplus \Z/3, & q = 2, \\
\Z^{\beta_2 +1}, & q = 1,  \\ 
\end{cases} $$\normalsize
where the Betti number $\beta_2$ is given in proposition~{\rm \ref{non-EuclideanPrincipalE2page}}.
\end{proposition}

The results of this proposition have been computed with HAP \cite{HAP} from the cell complex information computed with \cite{BianchiGP}.

\section{\emph{K}-theory} \label{K-theory} 
With the above information about the action of the \mbox{Bianchi} groups, we can, analogously to the way of \cite{Sanchez-Garcia}, compute the Bredon homology of the \mbox{Bianchi} groups, from which we can deduce their equivariant $K$-homology. 
The results of the computations~\cite{RahmThesis} are the following.

\begin{theorem} \label{K-homology}
Let $\beta_2$ be the Betti number specified in proposition {\rm \ref{non-EuclideanPrincipalE2page}}. For $\ringO_{-m}$ principal, the equivariant $K$-homology of $\Gamma := \mathrm{PSL}_2(\ringO_{-m})$ is isomorphic to \small
$$ \begin{array}{l|cccccc}

                                &m=1   & m=2            & m=3              & m=7  & m= 11            & m \in \{19,43,67,163\} \\ 

\hline &&&&&& \\

K^\Gamma_0(\underbar{E} \Gamma) & \Z^6 &\Z^5 \oplus \Z/2 \Z& \Z^5 \oplus \Z/2 \Z & \Z^3 & \Z^4 \oplus \Z/2 \Z & \Z^{\beta_2} \oplus \Z^3 \oplus \Z/2 \Z \\ 

\\

K^\Gamma_1(\underbar{E} \Gamma) & \Z   & \Z^3           &  0               & \Z^3 &  \Z^3            & \Z \oplus \Z^{\beta_2 +1}. \\ 

   \end{array}$$ \normalsize 
\end{theorem}

The remainder of the equivariant $K$-homology of $\Gamma$ is given by Bott $2$-periodicity.
By the Baum/Connes conjecture, which holds for the \mbox{Bianchi} groups \cite{JulgKasparov},
we obtain the $K$-theory of the reduced $C^*$-algebras of the \mbox{Bianchi} groups as isomorphic images.

\section{Appendix: The low terms of a free resolution \mbox{for the alternating group on 4 objects}}
\label{Appendix: The low terms of a free resolution for the alternating group on 4 objects}

We will use Wall's Lemma to construct a free resolution for $\Afour$, and compute its three differentials of lowest degrees explicitly.
This resolution will help us determine the maps induced on homology by inclusions into vertex stabilisers of type $\Afour$.

So let us recall Wall's lemma.
Given a group extension $1 \to K \to G \to H \to 1$,
and free resolutions $B$ for $K$, and $C$ for $H$, 
we construct a free resolution for $G$ in terms of the following double chain complex.
For any $s \in \N \cup \{0\}$, denote by $\alpha_s$ the number of generators of the free $\Z[H]$--module $C_s$. 
We define $D_s$ as the direct sum of $\alpha_s$ copies of $\Z[G] \otimes_K B_s$.
Then we have an augmentation of $D_s$ onto the direct sum of $\alpha_s$ copies of $\Z[H]$,
which we will identify with $C_s$, and write $\varepsilon_s: D_s \to C_s$. 
If $A_{r,s}$ is the submodule of $D_s$ which is the direct sum of $\alpha_s$ copies of 
$\Z[G] \otimes_K B_r$, then $A_{r,s}$ is a free $G$-module, and $D_s$ is the direct sum of the $A_{r,s}$.

\begin{Lem}[C.T.C. Wall \cite{Wall}]
There exist $G$-maps $d^k_{r,s} : A_{r,s} \to A_{r+k-1,s-k}$ for $k \geq 1$, $s \geq k$ such that
\begin{itemize}
\item $\varepsilon_{s-1} \circ d^1_{0,s} = d^C_s \circ \varepsilon_s : A_{0,s} \to C_{s-1}$ where $d^C$ denotes the differential in $C$,
\item $\sum\limits_{i=0}^k d^{k-i} \circ d^i = 0$, \qquad for each $k$,
where $d^k_{r,s}$ is interpreted as zero if $r = k = 0$, or if $s < k$.
\end{itemize}

\end{Lem}
Finally, we let $A$ denote the direct sum of the $A_{r,s}$ graded by dim$A_{r,s} = r +s$, and let $d := \sum\limits_k d^k$.
\begin{thm}[C.T.C. Wall \cite{Wall}]
$(A,d)$ is acyclic, and so yields a free resolution for $G$.
\end{thm}

Let $t$ be a generator of $\Z/n$. Let $F(n)$ be the periodic resolution of $\Z$ over $\Z[\Z/n]$ given by 
$$\xymatrix{ \ldots \ar[r]^{t-1 \qquad}  & \Z[\Z/n] \ar[rr]^{t^{n-1}+\ldots+t+1} & & \Z[\Z/n] \ar[r]^{t-1}  & \Z[\Z/n] \ar[rrr]^{\rm augmentation} & & & \Z.}$$
We consider the group extension $1 \to \Kleinfourgroup \to \Afour \to \Z/3 \to 1$, 
the resolution $F(3)$ for $\Z/3$ and the resolution $F(2) \otimes F(2)$ for $\Kleinfourgroup$.
Then, $A_{r,s} = \Z[\Afour] \otimes_{\Z[\Kleinfourgroup]} (\Z[\Kleinfourgroup])^r \cong (\Z[\Afour])^r$.
\\
Let us use the cycle type notation for the elements in the alternating group on four letters.
Then in low degrees, the differential of $F(2) \otimes F(2)$ becomes 
$$d^0_{1,s} = \left( (12)(34) -1, \quad (14)(23) -1 \right),$$
$$d^0_{2,s} = \begin{pmatrix} (12)(34) +1 & 1 -(14)(23) & 0 \\ 0 & (12)(34) -1 & (14)(23) +1 \end{pmatrix},$$
$$d^0_{3,s} = \begin{pmatrix} (12)(34) -1 & (14)(23) -1 & 0 & 0 \\ 0 & -(12)(34) -1 & (14)(23) +1 & 0 \\ 0 & 0 & (12)(34) -1 & (14)(23) -1 \end{pmatrix},$$
for all $s \in \N$.
At the same time, we can set $d^1_{0,2k} = \left( (132) + (123) +1 \right)$ 
and \\ \mbox{$d^1_{0,2k+1} = \left( (123) -1 \right)$} for all $k \in \N$, which satisfies the first condition in Wall's Lemma. 
Further, we set $$d^1_{1,1} = \begin{pmatrix} 1 & (142) \\ -(123) & (134)+1 \end{pmatrix},$$
$$d^1_{1,2} = \begin{pmatrix} -1 -(123) +(134) -(124) & (234) -(123) \\ (142)-(132) & (142)-1+(124)+(143) \end{pmatrix},$$
$$d^1_{2,1} = \begin{pmatrix} -1 & 0&  -(123) \\ 0 & (142)-1 & (243) \\ (134) & 0 & -(123)-1 \end{pmatrix}.$$

We sum up, and obtain the low degree terms of a free resolution for $\Afour$:
$$\xymatrix{ \ldots \ar[r]  & (\Z[\Afour])^{10} \ar[rr]^{d_3} & & (\Z[\Afour])^6 \ar[rr]^{d_2} & & (\Z[\Afour])^3  \ar[rr]^{d_1} & & \Z[\Afour] \to 0,}$$
where $d_1 = (d^1_{0,1}, \quad d^0_{1,0}) = \left( (123)-1, \quad (12)(34) -1, \quad (14)(23)-1 \right),$
\begin{center} $d_2 = $  $\begin{pmatrix} 
d^1_{0,2} & d^0_{1,1} & 0  \\ 
0 & d^1_{1,1} &   d^0_{2,0}  \end{pmatrix}$  
\mbox{$ = $ \small $\begin{pmatrix} (132) +(123)+1 & (12)(34) -1 & (14)(23)-1 & 0 & 0& 0 \\
0 & 1 & (142) & (12)(34) -1 & (14)(23) -1 & 0  \\ 
0 & -(123) & (134)+1 & 0 & (12)(34) -1 & (14)(23) -1 \end{pmatrix},$}  \normalsize \end{center}

\begin{center}and we assemble analogously $d_3 = $  $\begin{pmatrix} 
d^1_{0,3} & d^0_{1,2} & 0  & 0\\ 
0 & d^1_{1,2} &   d^0_{2,1} & 0\\
0 & d^2_{1,2} & d^1_{2,1} & d^0_{3,0} \end{pmatrix}$. 
\end{center}

\bibliographystyle{amsplain}
\bibliography{homologicalTorsionR}

\end{document}